\documentclass[letterpaper,11pt]{amsart}
\usepackage{amsmath}
\usepackage{amssymb}
\usepackage{amscd}
\usepackage{amsthm}
\usepackage{enumerate}
\usepackage{fullpage}
\usepackage{mathrsfs}
\usepackage{setspace}
\usepackage{mathtools}
\usepackage{xcolor}
\usepackage[all]{xy}
\usepackage{url}
\usepackage[colorlinks=true,linkcolor=blue,breaklinks,pagebackref]{hyperref}
\usepackage[lite,abbrev,alphabetic]{amsrefs}
\usepackage{braket}

\theoremstyle{definition}
\newtheorem{definition}{Definition}[section]

\theoremstyle{plain}
\newtheorem{theorem}[definition]{Theorem}
\newtheorem{proposition}[definition]{Proposition}
\newtheorem{lemma}[definition]{Lemma}
\newtheorem{corollary}[definition]{Corollary}

\newtheorem*{theorem*}{Theorem}

\theoremstyle{remark}
\newtheorem{remark}[definition]{Remark}

\def\N{\mathbb{N}}
\def\C{\mathbb{C}}
\def\R{\mathbb{R}}
\def\Q{\mathbb{Q}}
\def\Z{\mathbb{Z}}
\def\H{\mathbb{H}}

\def\GL{\mathrm{GL}}

\def\PSL{\mathrm{PSL}}

\def\SO{\mathrm{SO}}

\def\SU{\textnormal{SU}}

\def\GSp{\mathrm{GSp}}

\def\gl{\mathfrak{gl}}
\def\sl{\mathfrak{sl}}
\def\so{\mathfrak{so}}

\DeclareMathOperator{\diag}{diag} 

\DeclareMathOperator{\Ind}{Ind}

\DeclareMathOperator{\Hom}{Hom}

\DeclareMathOperator{\Irr}{Irr}

\DeclareMathOperator{\sgn}{sgn}

\DeclareMathOperator{\re}{Re}

\DeclareMathOperator{\im}{Im}

\DeclareMathOperator{\Ad}{Ad}
\DeclareMathOperator{\ad}{ad}
\DeclareMathOperator{\Kfin}{\text{$K$-fin}}
\DeclareMathOperator{\Ker}{Ker}
\DeclareMathOperator{\triv}{triv}
\DeclareMathOperator{\tr}{tr}
\DeclareMathOperator{\id}{id}
\DeclareMathOperator{\Sym}{Sym}
\DeclareMathOperator{\cpt}{cpt}

\newcommand{\fa}{\mathfrak{a}}
\newcommand{\fe}{\mathfrak{e}}
\newcommand{\ff}{\mathfrak{f}}
\newcommand{\fg}{\mathfrak{g}}
\newcommand{\fh}{\mathfrak{h}}
\newcommand{\fii}{\mathfrak{i}}
\newcommand{\fj}{\mathfrak{j}}
\newcommand{\fk}{\mathfrak{k}}
\newcommand{\fl}{\mathfrak{l}}
\newcommand{\fm}{\mathfrak{m}}
\newcommand{\fn}{\mathfrak{n}}
\newcommand{\fp}{\mathfrak{p}}
\newcommand{\fr}{\mathfrak{r}}
\newcommand{\fs}{\mathfrak{s}}
\newcommand{\fu}{\mathfrak{u}}
\newcommand{\fv}{\mathfrak{v}}
\newcommand{\fz}{\mathfrak{z}}

\newcommand{\fI}{\mathfrak{I}}
\newcommand{\fS}{\mathfrak{S}}
\newcommand{\fT}{\mathfrak{T}}
\newcommand{\fX}{\mathfrak{X}}

\newcommand{\cA}{\mathcal{A}}
\newcommand{\cJ}{\mathcal{J}}
\newcommand{\cO}{\mathcal{O}}
\newcommand{\cS}{\mathcal{S}}
\newcommand{\cU}{\mathcal{U}}
\newcommand{\cV}{\mathcal{V}}

\newcommand{\bfi}{\mathbf{i}}
\newcommand{\bfj}{\mathbf{j}}
\newcommand{\bfk}{\mathbf{k}}

\numberwithin{equation}{section}

\title{Epsilon dichotomy for linear models: the Archimedean case}

\author{Miyu Suzuki}

\author{Hiroyoshi Tamori}

\address{Miyu Suzuki \\
Department of Mathematics\\
Kyoto University\\
Kitashirakawa Oiwake-cho, Sakyo-ku, Kyoto 606-8502, Japan}
\email{suzuki.miyu.4c@kyoto-u.ac.jp}

\address{Hiroyoshi Tamori \\
Department of Mathematical Sciences\\
Shibaura Institute of Technology,\\
307 Fukasaku, Minuma-ku, Saitama City, Saitama, 337-8570, Japan}
\email{tamori@shibaura-it.ac.jp}

\begin{document}

\maketitle

\begin{abstract}
Let $G=\GL_{2n}(\R)$ or $G=\GL_n(\H)$ and $H=\GL_n(\C)$ regarded as a subgroup of $G$.
Here,  $\H$ is the quaternion division algebra over $\R$.
For a character $\chi$ on $\C^\times$,  we say that an irreducible smooth admissible moderate growth representation $\pi$ of $G$ is $\chi_H$-distinguished if $\Hom_H(\pi,  \chi\circ\det_H)\neq0$.
We compute the root number of a $\chi_H$-distinguished representation $\pi$ twisted by the representation induced from $\chi$.
This proves an Archimedean analogue of the conjecture by Prasad and Takloo-Bighash (J. Reine Angew. Math., 2011).  
The proof is based on the analysis of the contribution of $H$-orbits in a flag manifold of $G$ to the Schwartz homology of principal series representations. 
A large part of the argument is developed for general real reductive groups of inner type.
In particular,  we prove that the Schwartz homology $H_\ast(H,  \pi\otimes\chi)$ is finite-dimensional and hence it is Hausdorff for a reductive symmetric pair $(G, H)$ and a finite-dimensional representation $\chi$ of $H$.
\end{abstract}

\tableofcontents

\section{Introduction}
\label{sec:intro}

\subsection{Main results}
\label{sec:intro-1}

Let $D=\R$ or $D=\H$,  where $\H$ is the quaternion division algebra and $n$ be a positive integer.
Put 
    \[
    G=G_n=
        \begin{cases}
        \GL_{2n}(\R)  & \text{if $D=\R$} \\
        \GL_n(\H)  & \text{if $D=\H$},
        \end{cases} \quad\quad
    \varepsilon(D)=
        \begin{cases}
        -1  & \text{if $D=\R$} \\
        1  & \text{if $D=\H$}.
        \end{cases}    
    \]
Let $H=H_n$ be a subgroup of $G$ which is isomorphic to $\GL_n(\C)$.
Such a subgroup is unique up to conjugation by the Skolem-Noether theorem.

Let $\chi$ be a character on $\C^\times$.
For $F=\R,\C,\H$, we write $\det={\det}_{\GL_n(F)}$ for the determinant map on $\GL_n(\R)$ when $F=\R,\C$, and for the reduced norm on $\GL_n(\H)$ when $F=\H$.
We put $\chi_{\GL_n(F)}:=\chi\circ\det$.

By an \emph{SAF representation} of $G$, we mean a smooth admissible moderate growth Fr\'echet representation (cf. \cite{BK14}).
We say that an SAF representation $\pi$ of $G$ is \emph{$(H,  \chi_H)$-distinguished} (or \emph{$\chi_H$-distinguished}) if there exists a nonzero continuous $H$-intertwining operator from $\pi$ to $\chi_H$,  i.e.  $\Hom_H(\pi,  \chi_H)\neq0$.
The space $\Hom_H(\pi,  \chi_H)$ is at most one-dimensional if $\pi$ is irreducible \cite{Lu22}.

The goal of this paper is to prove the following theorem.
This is an Archimedean counterpart of the conjecture of Prasad and Takloo-Bighash \cite{PTB11} about representations of $p$-adic groups.

\begin{theorem}\label{thm:main}
Let $\pi$ be an irreducible SAF representation of $G$ and $\phi_\pi\,\colon W_\R\rightarrow\GL_{2n}(\C)$ its $L$-parameter.
Let $\chi$ be a character on $\C^\times$.
If $\pi$ is $\chi_H$-distinguished,  then
\begin{enumerate}
\item\label{ind:main1}
the $L$-parameter $\phi_\pi$ takes values in $\GSp_{2n}(\C)$ with similitude factor $\chi|_{\R^\times}$.
In other words, there exists a symplectic form $\langle\cdot,\cdot\rangle$ on $\C^{2n}$ such that $\langle \phi_{\pi}(w)v_1,\phi_{\pi}(w)v_2\rangle = \chi|_{\R^\times}(w)\langle v_1,v_2\rangle$ for any $w\in W_\R$ and any $v_1,v_2\in\C^{2n}$.
Here,  $\chi|_{\R^\times}$ is regarded as a character on $W_\R$ via the reciprocity isomorphism $W_\R^{ab}\cong\R^\times$;
\item\label{ind:main2}
the root number satisfies 
$\varepsilon\left(\phi_\pi\otimes\Ind_{\C^\times}^{W_\R}(\chi^{-1})\right)=\varepsilon(D)^n\chi(-1)^n$. 
\end{enumerate}
Conversely,  if $\pi$ is an essentially square integrable representation and satisfies (1) and (2),  then $\pi$ is $\chi_H$-distinguished.
\end{theorem}

In the case of $p$-adic groups,  the first author \cite{Suz} reduced the conjecture to the case of essentially square integrable representations.
Xue \cite{Xue} and S\'echerre \cite{Sec20} proved a large part of the case of supercuspidal representations.
The first author and Xue \cite{SX} reduced the case of essentially square integrable representations to that of supercuspidal representations.
Note that the character $\chi$ is assumed to be trivial in these works.  

In the Archimedean case,  there are no supercuspidal representations and the assertion for essentially square integrable representations is easy to check.
Therefore,  the main step toward Theorem \ref{thm:main} is the reduction to the case of essentially square integrable representations.

For an irreducible SAF representation $\pi$ of $\GL_N(D)$,  there exist a parabolic subgroup $P$ of $\GL_N(D)$ with a Levi subgroup isomorphic to $\GL_{n_1}(D)\times\cdots\times\GL_{n_r}(D)$ and irreducible essentially square integrable representations $\pi_i$ of $\GL_{n_i}(D)$ for each $i$ with a certain condition on central characters such that $\pi$ is a unique irreducible quotient of $\pi_1\times\cdots\times\pi_r$.
Here,  the product $\pi_1\times\cdots\times\pi_r$ means the normalized parabolic induction $\Ind_P^{\GL_N(D)}(\pi_1\boxtimes\cdots\boxtimes\pi_r)$.
Such a parabolically induced representation is called the standard module above $\pi$.

The idea is to find a necessary condition on $\pi_i$'s for the representation $\pi$ to be $\chi_H$-distinguished.
Then we deduce the conditions \ref{ind:main1} and \ref{ind:main2} in Theorem \ref{thm:main} from those on $\pi_i$'s.
The precise statement  of the necessary condition we need is as follows.

\begin{theorem}[Theorem \ref{thm:key}]\label{thm:key-intro}
Let $\pi$ and $\chi$ be as in Theorem \ref{thm:main} and let $\widetilde{\pi}=\pi_1\times\pi_2\times\cdots\times\pi_r$ be a standard module above $\pi$.
Here,  each $\pi_i$ is an irreducible essentially square integrable representation of $\GL_{n_i}(D)$.
Suppose that $\pi$ is $\chi_H$-distinguished.
Then there exists an involutive permutation $\varsigma\in\fS_r$ such that $n_{\varsigma(i)}=n_i$ and
    \[
        \begin{cases}
        \text{$n_i=2$ when $D=\R$ and $\pi_i$ is $\chi_{\GL_1(\C)}$-distinguished} 
        & \text{if $\varsigma(i)=i$,} \\
        \pi_{\varsigma(i)}\cong \pi_i^\vee
        \cdot\chi_{\GL_{n_i}(D)}
         & \text{if $\varsigma(i)\neq i$.}
        \end{cases}
    \]
\end{theorem}
This is an Archimedean analogue of \cite[Theorem 1.3]{Suz}.
As a byproduct of Theorem \ref{thm:key-intro},  we get the following corollary. See Section \ref{sec:5-4} for the proof.

\begin{corollary}\label{cor:dual}
Let $\pi$ be an irreducible SAF representation of $G$ and $\chi$ a character on $\C^\times$.
If $\pi$ is $\chi_H$-distinguished,  then 
    \[
    \pi\cong\pi^\vee\cdot\chi_G.
    \]
\end{corollary}
This is a generalization of the self-duality of distinguished representations when $\chi$ is trivial \cite[Theorem 6.7]{BM19}.

\subsection{Outline of the proof}
\label{sec:intro-2}

Let the notation be as in Theorem \ref{thm:key-intro}.
Write $\widetilde{\pi}$ as a parabolically induced representation $\Ind_P^{\GL_N(D)}(\gamma)$ from a standard parabolic subgroup $P$,  where $\gamma=\pi_1\boxtimes\cdots\boxtimes\pi_r$ is an essentially square integrable representation of a Levi subgroup $L$.
We fix a Cartan involution of $G$ and we may assume that $P$ and $L$ are standard with respect to it.
Let $X=G/P$ be the flag manifold and $\cV:=G\times_P(\gamma\cdot\delta_P^{1/2})$ the homogeneous tempered bundle on $X$ (c.f.  \cite[Proposition 6.7]{CS21}).
Here,  $\delta_P$ is the modular character of $P$.
Then,  the representation $\widetilde{\pi}$ is realized on the space $\cS(X,  \cV)$ of the Schwartz sections of $\cV$.

We take an decreasing sequence of open subsets $U_0:=X\supset U_1 \supset \cdots \supset U_{j_0-1} \supset U_{j_0}$,  so that $j_0$ is the number of $H$-orbits in $X$ and $U_j\setminus U_{j+1}$ is an $H$-orbit for each $j=0,  \ldots,  j_0-1$. 
Set $\Pi_j:=\cS(U_j,  \cV|_{U_j})$,  the space of Schwartz sections of the restriction to $U_j$ of $\cV$.
Then each $\Pi_j$ is a smooth Fr\'echet representation of moderate growth (we say SF representation for short) of $H$,  $\Pi_0=\widetilde{\pi}|_H$ and we obtain a descending sequence of $H$-subrepresentations
    \[
    \Pi_0\cdot\chi_H^{-1} \supset \Pi_1\cdot\chi_H^{-1}\supset
     \cdots \supset \Pi_{j_0}\cdot\chi_H^{-1}
    \]
of $(\widetilde{\pi}|_H)\cdot\chi_H^{-1}$.
In the case of $p$-adic groups, we can write the successive quotients $\rho_j:=\Pi_{j-1}/\Pi_j$ explicitly as compactly induced representations.
Hence the problem quickly reduces to the distinction of essentially square integrable representations.
This is the content of \cite{Suz}.
However,  in the current situation,  this is not the case and we must carefully analyze $\rho_j$ taking into account their topology.  
To that end,  we use the theory of Schwartz homology introduced by \cite{CS21}.
Our proof is inspired by \cite{Xue2}.

The Schwartz homology $H_\ast(H,  \Pi)$ of an SF representation $\Pi$ of $H$ is the derived functor of the coinvariant functor 
    \[
    \Pi \mapsto H_0(H,  \Pi)=\Pi_H:=\Pi/\sum_{h\in H}(h-1)\Pi.
    \]
They are linear topological spaces, which are not necessarily Hausdorff.
In particular,  the continuous dual of $H_0(H,  \Pi\cdot\chi_H^{-1})$ is $\Hom_H(\Pi,  \chi_H)$.
Write the $H$-orbit $U_{j-1}\setminus U_j$ as $Hg_jP/P$ with some $g_j\in G$ and set $Q_j:=H\cap g_jPg_j^{-1}$.
Let $\fg$,  $\fh$ and $\fp$ be the Lie algebras of $G$,  $H$ and $P$,  respectively.
Then,  each $\rho_j$ has a descending sequence of $H$-subrepresentations $\rho_j=\rho_{j,  0} \supset \rho_{j,  1} \supset \rho_{j,  2} \supset \cdots$ such that
    \begin{equation}\label{eq:Borel-intro}
    \rho_j\cong \varprojlim_k \rho_j/\rho_{j,  k},  \quad\quad
    \rho_{j,  k}/\rho_{j,  k+1}\cong \cS(H/Q_j,  
    H\times_{Q_j}({}^{g_j^{-1}}(\gamma\cdot\delta_P^{1/2})\otimes 
    \Sym^k(\fg/\fs_j)_\C^\vee)),
    \end{equation}
where $\fs_j=\fh+\Ad(g_j)\fp$ and $H\times_{Q_j}(\fg/\fs_j)_\C^\vee$ is the complexification of the conormal bundle for the closed subset $Hg_jP/P$ of $U_{j-1}$.

The short exact sequences $0 \rightarrow \Pi_j \rightarrow \Pi_{j-1} \rightarrow \rho_j \rightarrow 0$
and $0 \rightarrow \rho_{j, k+1} \rightarrow \rho_{j,  k} \rightarrow 
    \rho_{j,  k}/\rho_{j,  k+1} \rightarrow 0$ yield long exact sequences of Schwartz homologies
    \[
    \cdots \rightarrow H_l(H,  \Pi_j\cdot\chi_H^{-1}) \rightarrow 
    H_l(H,  \Pi_{j-1}\cdot\chi_H^{-1}) \rightarrow 
    H_l(H,  \rho_j\cdot\chi_H^{-1}) \rightarrow 
    H_{l-1}(H,  \Pi_j\cdot\chi_H^{-1}) \rightarrow \cdots,
    \]
and 
    \[
    \cdots \rightarrow H_l(H,  \rho_{j,  k+1}\cdot\chi_H^{-1}) \rightarrow 
    H_l(H,  \rho_{j,  k}\cdot\chi_H^{-1}) \rightarrow 
    H_l(H,  (\rho_{j,  k}/\rho_{j,  k+1})\cdot\chi_H^{-1}) \rightarrow 
    H_{l-1}(H,  \rho_{j,  k+1}\cdot\chi_H^{-1}) \rightarrow \cdots.
    \]
Set $d_{l,  j,  k}=\dim H_l(H,  (\rho_{j,  k}/\rho_{j,  k+1})\cdot\chi_H^{-1})$.
From \eqref{eq:Borel-intro} and the above two long exact sequences,  at least formally we obtain 
    \begin{equation}\label{eq:upbdd-intro}
    \dim H_l(H,  \Pi_0\cdot\chi_H^{-1}) \leq 
    \sum_{j=1}^{j_0} \sum_{k=0}^\infty d_{l,  j,  k}.
    \end{equation}
We show that $d_{l,  j,  k}<\infty$ for any $l,  j,  k$ and $d_{l,  j,  k}=0$ except for finitely many $k$.
Thus the right hand side of \eqref{eq:upbdd-intro} makes sense and this inequality holds.
These arguments are developed for more general setting to obtain $\dim H_l(H,  \Pi\otimes\chi)<\infty$ for any reductive symmetric pair $(G, H)$, any finite length SAF representation $\Pi$ of $G$ and any finite-dimensional SF representation $\chi$ of $H$.
By \cite[Proposition 1.9]{CS21},  this proves that $H_l(H,  \Pi\otimes\chi)$ is Hausdorff.

Let $\sigma$ be the involution on $G$ so that $H=G^\sigma$ and set $\sigma_j=\Psi(g_j)^{-1} \circ \sigma \circ \Psi(g_j)$.
Here,  $\Psi(g_j)$ denotes the inner automorphism of $G$ attached to $g_j$.
Then we prove the following two vanishing results on $d_{l,  j,  k}$:
\begin{itemize}
\item For any $j$ and $k>0$,  $d_{0,  j,  k}=0$. 
\item If $\sigma_j(L) \neq L$,  then $d_{0,  j,  k}=0$,  where $L$ is the standard Levi subgroup of $P$.
\end{itemize}
Hence \eqref{eq:upbdd-intro} for $l=0$ becomes
    \[
    \dim\Hom_H(\Pi_0,  \chi_H)=\dim H_0(H,  \Pi_0\cdot\chi_H^{-1})\leq 
    \sum_{j: \sigma_j(L)=L} d_{0,  j,  0}.
    \]
One can see that $H_0(H,  (\rho_{j,  0}/\rho_{j,  1})\cdot\chi_H^{-1})$ is isomorphic to $H_0(L^{\sigma_j},  \gamma\cdot{}^{g_j}\chi_H^{-1})$ if $\sigma_j(L)=L$.
Hence the right hand side of the above inequality is the sum of $d_{0,  j,  0}=\dim\Hom_{L^{\sigma_j}}(\gamma,  {}^{g_j}\chi_H)$.

Now we write $L=\GL_{n_1}(D) \times\GL_{n_2}(D) \times\cdots\times\GL_{n_r}(D)$.
Note that $n_i\in\{1,  2\}$ and $n_1+\cdots+n_r=2n$ if $D=\R$ and $r=n$ and $n_i=1$ for all $i$ if $D=\H$.
By the explicit choice of the representatives $\{g_j\}_j$ given by \cite{Cho19},  we obtain the following description on $\sigma_j|_L$ which preserves $L$.
There exists an involution $\varsigma\in\fS_r$ such that $n_{\varsigma(i)}=n_i$ for each $i$,  $n_i=2$ if $\varsigma(i)=i$ and $D=\R$, and 
    \[
    \sigma_j(x_1,  x_2,  \ldots,  x_r)=
    (\upsilon_1(x_{\varsigma(1)}),  \upsilon_2(x_{\varsigma(2)}),  \ldots,  
    \upsilon_r(x_{\varsigma(r)})), \quad\quad x_i\in\GL_{n_i}(D).
    \]
Here,  $\upsilon_i$ is an inner automorphism of $\GL_{n_i}(D)$ which satisfies $\upsilon_i^2=\id$ and $\GL_{n_i}(D)^{\upsilon_i}$ is isomorphic to $\GL_1(\C)$ if $\varsigma(i)=i$.
It is also checked that the character $({}^{g_j}\chi_H)|_{L^{\sigma_j}}$ is given by
    \[
    {}^{g_j}\chi_H(x_1,  \ldots,  x_r)=
    \prod_{\varsigma(i)<i}\chi_{\GL_{n_i}(D)}(x_i)
    \prod_{\varsigma(i)=i}\chi_{\GL_1(\C)}(x_i).
    \]
Therefore $\Hom_{L^{\sigma_j}}(\gamma,  {}^{g_j}\chi_H)\neq0$ if and only if
    \[
        \begin{cases}
        \text{$\pi_i$ is $\chi_{\GL_1(\C)}$-distinguished} 
        & \text{if $\varsigma(i)=i$} \\
        \pi_{\varsigma(i)}\cong \pi_i^\vee
        \cdot\chi_{\GL_{n_i}(D)}
         & \text{if $\varsigma(i)\neq i$}.
        \end{cases}
    \]
This proves Theorem \ref{thm:key-intro}.

This paper is organized as follows.
In Section \ref{sec:2},  we recall the local Langlands correspondence for $\GL_N(D)$ and deduce Theorem \ref{thm:main} from Theorem \ref{thm:key-intro}.
Section \ref{sec:3} is devoted to recall the necessary definitions and facts about Schwartz homology and relative Lie algebra homology.
Several lemmas about the structure of symmetric pairs are prepared in Section \ref{sec:4}.
In Section \ref{sec:5-1},  we obtain an upper bound for the dimension of the Schwartz homology and in Section \ref{sec:5-2},  we prove $\dim H_\ast(H, \pi\otimes\chi)<\infty$ for an SAF representation $\pi$ of finite length of a real reductive group $G$, a symmetric subgroup $H$ and a finite-dimensional representation $\chi$ of $H$.
We show in Section \ref{sec:5-3} that under certain conditions,  the (normal derivatives for) $H$-orbits in a flag manifold do not contribute to the upper bound for $\dim H_0(H,  \pi)$.
Using this fact,  we prove Theorem \ref{thm:key-intro} in Section \ref{sec:5-4}.

\subsection{Notations}\label{sec:intro-3}
In this section, we introduce notations which are used throughout this paper.

For a positive integer $n$, set $[1,n]:=\{1,2,\ldots,n\}$ and let $\fS_n$ be the symmetric group of degree $n$.
For a complex number $z$,  let $\bar{z}$ be its complex conjugate  and $|z|=\sqrt{z\bar{z}}$ the usual absolute value.
The unit circle of the norm 1 complex numbers is denoted as $\C^1$.
Let $\H=\R+\R\bfi+\R\bfj+\R\bfk$ denote Hamilton's quaternion division algebra.
The main involution (conjugation) on $\H$ is denoted by $z\mapsto\bar{z}$.
Let $\sgn$ denote the sign character on $\R^\times$.
We regard it as a character on $G_n$ by composing with ${\det}_{G_n}$.

Lie groups will be denoted by Latin capital letters and their Lie algebras by corresponding lower case German letters.
We put subscript $\C$ to denote their complexifications.
Given an involution $\sigma$ on a Lie group $G$, we use the same letter to write the corresponding involution on $\fg$, and the one on the root system of $\fg$ with respect to a $\sigma$-invariant Cartan subalgebra.
We write $\fg^{\sigma}$, $\fg^{-\sigma}$ for the eigenspace of $\sigma$ with eigenvalue $1$, $-1$, respectively.
If $\fg$ is abelian and $V$ is a semisimple $\fg$-module, we also use $\sigma$ to write the corresponding involution on the set $\Sigma$ of weights in $V$, and write $\Sigma^{\sigma}$, $\Sigma^{-\sigma}$ for the set of fixed weights by $\sigma$, $-\sigma$, respectively, and put $\rho_V:=\frac{1}{2}\tr(\ad(\cdot)|V)\in\fg^{\vee}$. Here $\fg^{\vee}$ denotes the linear dual of $\fg$.

We write $\fz(\fg), \fz_{\fg}(\fs)$ for the center of $\fg$, the centralizer of a subspace $\fs$ in $\fg$, respectively.

For $g\in G$,  let $\Psi(g)\,\colon x\mapsto gxg^{-1}$ be the inner automorphism of $G$.
Given a representation $\pi$ of $G$ and $g\in G$, we write ${}^g\pi$ for the representation of $G$ where the representation space equals the one of $\pi$ and the action is given by ${}^g\pi(h)=\pi(\Psi(g)(h))=\pi(ghg^{-1})$ for any $h\in G$.

By \emph{a real reductive group},
we use the definition of \cite[Section 2.1.1]{Wal84}.
We call a symmetric pair $(G,H)$ \emph{reductive} if $G$ is real reductive.
For a parabolic subgroup $P$ of a real reductive group $G$ with a Langlands decomposition $P=MAN$, let $\bar{P}$ be the opposite of $P$ such that $P\cap\bar{P}=MA$ and $\bar{N}$ denote the unipotent radical of $\bar{P}$.

Smooth Fr{\'e}chet representations of moderate growth are called 
SF representations for short. 
Similarly, smooth admissible Fr{\'e}chet representations of moderate growth are called SAF representations.
Recall that an SAF globalization of a Harish-Chandra module $V$ is unique up to isomorphisms and is a nuclear Fr\'echet space (see \cite{BK14} for example). It is called the Casselman-Wallach globalization of $V$.

For a real reductive group $G$ and an SAF representation $\pi$ of finite length of $G$, we write $\pi^{\vee}$ for the contragredient SAF representation, that is, the Casselman-Wallach globalization of the contragredient of the underlying $(\fg_{\C},K)$-module of $\pi$ for a maximal compact subgroup $K$ of $G$.

Write $(\SO(2)^{\sim})^\wedge$ for the unitary dual of the universal cover $\SO(2)^{\sim}$ of the special orthogonal group $\SO(2)$.
We fix an isomorphism $\R\cong (\SO(2)^{\sim})^\wedge; t\mapsto \xi_t$ of topological groups so that $\sl_2(\R)\cong \xi_{-2}\oplus\xi_0\oplus\xi_2$ via the adjoint action.

\section{Reduction}
\label{sec:2}

\subsection{Langlands correspondence for $\GL_N(D)$}
Let $N$ be a positive integer.
In this subsection,  we summarize basic facts about $L$-parameters and root numbers of representations of $\GL_N(D)$.
For details,  see \cite{K94},  \cite{J79} and \cite{T79}.

First we recall the classification of semisimple representations of \emph{the Weil group} of $\R$.
The Weil group $W_\R$ of $\R$ is the non-split extension of $\C^\times$ by $\Z/2\Z$:
    \[
    W_\R=\C^\times\sqcup\C^\times j,  \quad\quad jzj^{-1}=\bar{z},  \quad j^2=-1.
    \]
The commutator subgroup $[W_\R,  W_\R]$ is $\C^1$ and the abelianization $W_\R^{ab}=W_\R/[W_\R,  W_\R]$ is identified with $\R^\times$ via \emph{the reciprocity map} $r_\R\,\colon\R^\times\xrightarrow{\sim} W_\R^{ab}$ given by
    \[
    r_\R(x)=
        \begin{cases}
        \sqrt{x}\cdot[W_\R,  W_\R] & (x>0), \\
        j\sqrt{-x}\cdot[W_\R,  W_\R] & (x<0).
        \end{cases}
    \]
Let $\Phi(W_\R)$ (resp.\,$\Irr(W_\R)$) be the set of equivalence classes of  finite-dimensional semisimple (resp.\,irreducible) complex representations of $W_\R$.

\begin{lemma}\label{lem:L-param}
For $k\in\Z$ and $\lambda\in\C$,  let $\theta_{k,  \lambda}$ be the character on $\C^\times$ given by $\theta_{k,  \lambda}(z)=(z/|z|)^k|z|^{2\lambda}$.
We denote the two-dimensional representation  $\Ind_{\C^\times}^{W_\R}(\theta_{k,  \lambda})$ of $W_\R$ as $\phi^{(2)}_{k,  \lambda}\in\Phi(W_\R)$.
For $k\in\{0,  1\}$ and $\lambda\in\C$,  let $\phi^{(1)}_{k,  \lambda}$ be the character on $W_\R$ given as $\phi^{(1)}_{k,  \lambda}(z)=|z|^{2\lambda}$ and $\phi^{(1)}_{k,  \lambda}(j)=(-1)^k$. 

\begin{itemize}
\item[(1)] The representation $\phi^{(2)}_{k,  \lambda}$ is reducible if and only if $k=0$.
When $k=0$,  $\phi^{(2)}_{0,  \lambda}=\phi^{(1)}_{0,  \lambda}\oplus\phi^{(1)}_{1,  \lambda}$.
\item[(2)] For $k,  k'\in\Z$ and $\lambda,  \lambda'\in\C$,  we have $\phi^{(2)}_{k,  \lambda}\cong\phi^{(2)}_{-k,  \lambda}$,  $\det\left(\phi_{k,  \lambda}^{(2)}\right)=\phi_{k+1,  2\lambda}^{(1)}$ and
    $\phi^{(2)}_{k,  \lambda}\otimes\phi^{(2)}_{k',  \lambda'}\cong
    \phi^{(2)}_{k+k',  \lambda+\lambda'}\oplus
    \phi^{(2)}_{k-k',  \lambda+\lambda'}$.
\item[(3)] Irreducible representations of $W_\R$ are of dimension one or two.
A one-dimensional representation $\phi$ is written uniquely as $\phi=\phi^{(1)}_{k,  \lambda}$ with $k\in\{0,  1\}$ and $\lambda\in\C$.
A two-dimensional irreducible representation $\phi$ is written uniquely as $\phi=\phi^{(2)}_{k,  \lambda}$ with $k\in\Z_{\geq1}$ and $\lambda\in\C$.
\item[(4)] The contragredient representations of $\phi^{(1)}_{k,\lambda}, \phi^{(2)}_{k,\lambda}$ are $\phi^{(1)}_{k,-\lambda}, \phi^{(2)}_{k,-\lambda}$, respectively.
\end{itemize}
\end{lemma}

Next,  we recall the classification of irreducible SAF representations of $\GL_N(D)$.
Let $\Irr(\GL_N(D))$ be the set of equivalence classes of irreducible SAF representations of $\GL_N(D)$.

Suppose that $D=\R$.
For $k\in\{0,  1\}$ and $\lambda\in\C$,  let $\pi^{(1)}_{k,  \lambda}$ denote the character $\sgn^k|\cdot|^\lambda$ on $\GL_1(\R)$.
For $k\in\Z_{\geq1}$ and $\lambda\in\C$, 
we define $\pi^{(2)}_{k,\lambda}$ to be the irreducible SAF representation of $\GL_2(\R)$ characterized by the following properties:
\begin{itemize}
\item
the central character of $\pi^{(2)}_{k,\lambda}$ equals $\sgn^{k+1}|\cdot|^{2\lambda}$,
\item
the space $(\pi^{(2)}_{k,\lambda})_{\text{$\SO(2)$-fin}}$ of $\SO(2)$-finite vectors is isomorphic to $\bigoplus_{\epsilon\in\{\pm1\}}\bigoplus_{l\in\N}\xi_{\epsilon(k+1+2l)}$ as representations of $\SO(2)$ (see Section \ref{sec:intro-3} for the definition of $\xi_t$).
\end{itemize}

\begin{lemma}\label{lem:irrepR}
{\rm (Classification for $\Irr(\GL_N(\R))$)}
\begin{enumerate}
\item Representations $\pi^{(1)}_{k,  \lambda}$ and $\pi^{(2)}_{k,  \lambda}$ are essentially square integrable representations.
Any irreducible essentially square integrable representation $\pi$ of $\GL_N(\R)$ is uniquely written as $\pi=\pi^{(m)}_{k,  \lambda}$ with $m\in\{1,  2\}$.
\item For general $N$,  let $(n_1,  \ldots,  n_r)$ be a partition of $N$ consisting of 1 or 2 and $\pi_j=\pi^{(n_j)}_{k_j,  \lambda_j}$ be an irreducible essentially square integrable representation of $\GL_{n_j}(\R)$.
If $\pi_1,  \ldots,  \pi_r$ satisfy 
\begin{align}\label{eq:irrepR}
n_1^{-1}\re(\lambda_1) \geq n_2^{-1}\re(\lambda_2) \geq \cdots \geq n_r^{-1}\re(\lambda_r),
\end{align}
then $\pi_1\times\pi_2\times\cdots\times\pi_r$ has a unique irreducible quotient $\pi=\pi_1\boxplus\pi_2\boxplus\cdots\boxplus\pi_r$.
Any irreducible SAF representation of $\GL_N(\R)$ is uniquely written in this way up to permutations of $r$ indices preserving \eqref{eq:irrepR}.
\end{enumerate}
\end{lemma}

For an irreducible SAF representation $\pi\in\Irr(\GL_N(\R))$,  its $L$-parameter $\phi_\pi\in\Phi(W_\R)$ is defined as follows.
We write $\pi=\pi_1\boxplus\pi_2\boxplus\cdots\boxplus\pi_r$ with essentially square integrable representations $\pi_j=\pi^{(n_j)}_{k_j,  \lambda_j}\in\Irr(\GL_{n_j}(\R))$ and set $\phi_\pi=\bigoplus_{j=1}^r\phi^{(n_j)}_{k_j,  \lambda_j}$.
The induced representation $\widetilde{\pi}=\pi_1\times\pi_2\times\cdots\times\pi_r$ in Lemma \ref{lem:irrepR} (2) is called a \emph{standard module} above $\pi$.

Next we suppose $D=\H$.
For $k\in\Z_{\geq1}$ and $\lambda\in\C$,  there is a unique irreducible representation $\tau_{k,  \lambda}\in\Irr(\GL_1(\H))$ of dimension $k$ with central character $\sgn^{k+1}|\cdot|^{2\lambda}$.
This is an essentially square integrable representation and any irreducible essentially square integrable representation of $\GL_1(\H)$ is written uniquely in this way.

\begin{lemma}\label{lem:irrepH}
{\rm (Classification for $\Irr(\GL_N(\H))$)}
Let $\tau_1=\tau_{k_1,  \lambda_1},  \ldots,  \tau_N=\tau_{k_N,  \lambda_N}$ be irreducible representations of $\GL_1(\H)$.
If they satisfy 
\begin{align}\label{eq:irrepH}
\re(\lambda_1) \geq \re(\lambda_2) \geq \cdots \geq \re(\lambda_N),
\end{align}
then $\tau_1\times\tau_2\times\cdots\times\tau_N$ has a unique irreducible quotient 
$\tau=\tau_1\boxplus\tau_2\boxplus\cdots\boxplus\tau_N$.
Any irreducible SAF representation of $\GL_N(\H)$ is uniquely written in this way up to permutations of $N$ indices preserving \eqref{eq:irrepH}.
\end{lemma}

For an irreducible representation $\tau\in\Irr(\GL_N(\H))$,  its $L$-parameter $\phi_\tau\in\Phi(W_\R)$ is defined as follows.
We write $\tau=\tau_1\boxplus\tau_2\boxplus\cdots\boxplus\tau_N$ with essentially square integrable representations $\tau_j=\tau_{k_j,  \lambda_j}\in\Irr(\GL_1(\H))$ and set $\phi_\tau=\bigoplus_{j=1}^N\phi^{(2)}_{k_j,  \lambda_j}$.
The induced representation $\widetilde{\tau}=\tau_1\times\tau_2\times\cdots\times\tau_N$ in Lemma \ref{lem:irrepH}  is called a \emph{standard module} above $\tau$.

We say that $\phi=\bigoplus_{j=1}^r\phi_j\in\Phi(W_\R)$ with $\phi_j\in\Irr(W_\R)$ is \emph{relevant to $\GL_N(D)$} if 
    \[
    \begin{cases}
    \dim\phi=N  & \text{when $D=\R$,} \\
    \dim\phi=2N, r=N \text{ and } \dim\phi_j=2 \text{ for all $j$} &
    \text{when $D=\H$.}
    \end{cases}
    \]
Let $\Phi(\GL_N(D))$ denote the set of $L$-parameters relevant to $\GL_N(D)$.
Now we can state the local Langlands correspondence for $\GL_N(D)$.

\begin{theorem}[Langlands correspondence for $\GL_N(D)$]\label{thm:LLC}
The map $\pi\mapsto\phi_\pi$ defines a bijection from $\Irr(\GL_N(D))$ to  $\Phi(\GL_N(D))$.
\end{theorem}

Let $\psi$ be a non-trivial additive character on $\R$ and take $a\in\R^\times$ so that $\psi(x)=\exp(2\pi\sqrt{-1}ax)$.
The $\varepsilon$-factor $\varepsilon(s,  \phi,  \psi)$ of $\phi\in\Phi(W_\R)$ is defined as follows.
When $\phi\in\Irr(W_\R)$,  set
    \[
    \varepsilon(s,  \phi,  \psi)=
        \begin{cases}
        (\sgn(a)\sqrt{-1})^k|a|^{\lambda+s-\frac12}
        & \text{if $\phi=\phi^{(1)}_{k,  \lambda}$ with $k\in\{1,  0\}$ and 
        $\lambda\in\C$} \\
        (\sgn(a)\sqrt{-1})^{k+1}|a|^{2(\lambda+s)-1}
       & \text{if $\phi=\phi^{(2)}_{k,  \lambda}$ with $k\in\Z_{\geq1}$ and 
        $\lambda\in\C$}
        \end{cases}
    \]
For general $\phi$,  let $\phi=\bigoplus_{j=1}^r\phi_j$ be the irreducible decomposition and set $\varepsilon(s,  \phi,  \psi)=\prod_{j=1}^r\varepsilon(s,  \phi_j,  \psi)$.
Assume that $\phi$ is a self-dual representation and $\det(\phi)$ is the trivial character.
Then $\varepsilon(\tfrac12,  \phi,  \psi)$ is independent of the choice of $\psi$ and satisfies $\varepsilon(\tfrac12,  \phi,  \psi)^2=1$ \cite[Proposition 5.1]{GGP}.
The value $\varepsilon(\tfrac12,  \phi,  \psi)$ is denoted as $\varepsilon(\phi)$.
It is called the \emph{root number} of $\phi$.

\begin{remark}
Let $\pi$, $\phi_\pi$ and $\chi$ be as in Theorem \ref{thm:main}.
If the condition \ref{ind:main1} is satisfied,  then the representation $\phi=\phi_\pi\otimes\Ind_{\C^\times}^{W_\R}(\chi^{-1})$ is self-dual and $\det(\phi)$ is the trivial character.
Thus the notation $\varepsilon\left(\phi_\pi\otimes\Ind_{\C^\times}^{W_\R}(\chi^{-1})\right)$ in the condition \ref{ind:main2} is justified.
\end{remark}

\subsection{Reduction to Theorem \ref{thm:key-intro}}\label{sec:2-2}
Retain notation in Section \ref{sec:intro-1}.
We show that Theorem \ref{thm:key-intro} implies Theorem \ref{thm:main}.

\begin{proof}[Proof of Theorem \ref{thm:main} assuming Theorem \ref{thm:key-intro}]
Let $\pi$ and $\chi$ be as in Theorem \ref{thm:main}.
Write $\chi=\theta_{k',  \lambda'}$ with $k'\in\Z$ and $\lambda'\in\C$ (see Lemma \ref{lem:L-param}).

First we consider the case where $\pi$ is an essentially square integrable representation.
If this is the case,  $n=1$ and the $L$-parameter $\phi_\pi$ is of the form $\phi_\pi=\phi_{k,  \lambda}^{(2)}$ with $k\in\Z_{\geq1}$ and $\lambda\in\C$.
Since $\phi_\pi$ takes values in $\GL_2(\C)=\GSp_2(\C)$ and its similitude factor equals $\det\left(\phi_\pi\right)=\phi_{k+1,  2\lambda}^{(1)}$,  the condition \ref{ind:main1} is equivalent to $\phi_{k+1,  2\lambda}^{(1)}\circ r_\R=\chi|_{\R^\times}$.
The left hand side is $\sgn^{k+1}|\cdot|^{2\lambda}$ and the right hand side is $\sgn^{k'}|\cdot|^{2\lambda'}$.
Hence the condition \ref{ind:main1} in this case is equivalent to $\lambda=\lambda'$ and $k-k'$ is odd.

Let $D=\R$ and $\pi=\pi^{(2)}_{k,\lambda}$. We may regard $H=\GL_1(\C)$ as the subgroup of $G=\GL_2(\R)$ generated by the central elements $\diag(a,a)$  $(a>0)$ and $\SO(2)$. Since the projection from $\pi$ to an $\SO(2)$-isotypic component is continuous, $\pi$ is $\chi_H$-distinguished if and only if $|a|^{2\lambda}=|a|^{2\lambda'}$ for any $a>0$ and the $\SO(2)$-type of $\pi$, which we saw before Lemma \ref{lem:irrepR}, contains $\xi_{k'}$. 
Hence $\pi$ is $\chi_H$-distinguished if and only if $\lambda=\lambda'$,  $|k'|> k$ and $k-k'$ is odd.

Let $D=\H$ and $\pi=\tau_{k,  \lambda}$. Write $\R_{>0}$ for the center of $G=\GL_1(\H)$. Then $H\cong\R_{>0}\times\SO(2)$ under an isomorphism $G\cong \R_{>0}\times\SU(2)$. The $k$-dimensional irreducible representation of $\SU(2)$ decomposes into $\bigoplus_{l}\xi_l$ as a representation of $\SO(2)$, where $l$ runs over integers satisfying $|l|\le k-1$ and $k-1-l\in 2\Z$.
Therefore $\pi$ is $\chi_H$-distinguished if and only if $\lambda=\lambda'$,  $|k'|< k$ and $k-k'$ is odd.

On the other hand,  it is easy to check that for $\psi(x)=\exp(2\pi\sqrt{-1}ax)$,
    \[
    \varepsilon(\tfrac12,  \phi_\pi\otimes\Ind_{\C^\times}^{W_\R}(\chi^{-1}),  \psi)=
        \begin{cases}
        (-1)^{k'+1}|a|^{4(\lambda-\lambda')} & \text{if $|k'|>k$}, \\
        (-1)^{k+1}|a|^{4(\lambda-\lambda')} & \text{if $|k'|<k$}
        \end{cases}
    \]
by Lemma \ref{lem:L-param} (2).
Therefore,  $\pi$ is $\chi_H$-distinguished if and only if the conditions \ref{ind:main1} and \ref{ind:main2} hold.

Next,  we treat the general case.
Suppose that $\pi$ is $\chi_H$-distinguished.
Let $\widetilde{\pi}=\pi_1\times\pi_2\times\cdots\times\pi_r$ be a standard module above $\pi$ and $\varsigma\in\fS_r$ the involution in Theorem \ref{thm:key-intro}.
Note that the condition $\pi_{\varsigma(i)}\cong \pi_i^\vee\cdot\chi_{\GL_{n_i}(D)}$ is equivalent to $\phi_{\pi_{\varsigma(i)}}=\phi_{\pi_i}^\vee\cdot\chi|_{\R^\times}$.
Hence the condition \ref{ind:main1} for the representation $\pi$ follows from that for essentially square integrable representations and the properties of the involution $\varsigma$.
Since $\phi_\pi=\bigoplus_{i=1}^r\phi_{\pi_i}$,  we have
    \[
    \varepsilon\left(\phi_\pi\otimes\Ind_{\C^\times}^{W_\R}(\chi^{-1})\right)=
    \prod_{i=1}^r \varepsilon\left(\tfrac12,  
    \phi_{\pi_i}\otimes\Ind_{\C^\times}^{W_\R}(\chi^{-1}),  \psi\right)
    \]
If $\varsigma(i)=i$,  then $\varepsilon\left(\tfrac12,  \phi_{\pi_i}\otimes\Ind_{\C^\times}^{W_\R}(\chi^{-1}),  \psi\right)=\varepsilon(D)\chi(-1)$ as we have seen above.
If $\varsigma(i)\neq i$,  then 
    \begin{align*}
    \phi_{\pi_{\varsigma(i)}}\otimes\Ind_{\C^\times}^{W_\R}(\chi^{-1})
    =\phi_{\pi_i}^\vee\cdot(\chi|_{\R^\times})
    \otimes\Ind_{\C^\times}^{W_\R}(\chi^{-1})
    =\left(\phi_{\pi_i}\otimes\Ind_{\C^\times}^{W_\R}(\chi^{-1})\right)^\vee
    \end{align*}
since $(\chi|_{\R^\times})\otimes\Ind_{\C^\times}^{W_\R}(\chi^{-1})\cong\Ind_{\C^\times}^{W_\R}(\chi)\cong\left(\Ind_{\C^\times}^{W_\R}(\chi^{-1})\right)^\vee$.
We get from \cite[(3.4.7)]{T79} (see also \cite[p.14]{GGP}),  
    \begin{gather*}
    \varepsilon\left(\tfrac12,  
    \phi_{\pi_{\varsigma(i)}}\otimes\Ind_{\C^\times}^{W_\R}(\chi^{-1}),  \psi\right)
    \varepsilon\left(\tfrac12,  
    \phi_{\pi_i}\otimes\Ind_{\C^\times}^{W_\R}(\chi^{-1}),  \psi\right)
    =\det\left(\phi_{\pi_i}\otimes\Ind_{\C^\times}^{W_\R}(\chi^{-1})\right)
    \circ r_\R(-1) \\
    =\begin{cases}
         \det\left(\Ind_{\C^\times}^{W_\R}(\chi^{-1})\right)\circ r_\R(-1)
         =\phi^{(1)}_{-k'+1,  -2\lambda'}(j)=-\chi(-1)
         & \text{if $\dim\phi_{\pi_i}=1$,} \\
         1 & \text{if $\dim\phi_{\pi_i}=2$.}
         \end{cases}
    \end{gather*}
Set 
    \begin{gather*}
    a=\frac12\cdot\#\{i \mid \varsigma(i)\neq i,  \quad\dim\phi_{\pi_i}=1\},  \quad
    b=\frac12\cdot\#\{i \mid \varsigma(i)\neq i,  \quad \dim\phi_{\pi_i}=2\},  \\
    c=\#\{i \mid \varsigma(i)=i
    \}.
    \end{gather*}
Then $a+2b+c=n$ and
    \[
    \varepsilon\left(\phi_\pi\otimes\Ind_{\C^\times}^{W_\R}(\chi^{-1})\right)
    =(\varepsilon(D)\chi(-1))^{a+c}=(\varepsilon(D)\chi(-1))^n.
    \]
\end{proof}

\section{Preliminaries on homology}
\label{sec:3}

\subsection{Schwartz homology}
Let $G$ be an almost linear Nash group and $\pi$ an SF representation of $G$.
Recall that a finite cover of an open subgroup of a linear algebraic group defined over $\R$ has a natural structure of an almost linear Nash group. 
In this section, we recall basic properties of the Schwartz homology $H_*(G,\pi)$. See \cite{CS21} for the precise definition and proofs. 

The Schwartz homology groups $H_*(G,\pi)$ are (not necessarily Hausdorff) linear topological spaces, and there is a natural identification \[H_0(G,\pi)=\pi_G:=\pi/\sum_{g\in G}(g-1)\pi,\] 
where $\pi_G$ is equipped with the quotient topology of $\pi$.

Let $K$ be a maximal compact subgroup of $G$. 
The space $\pi_{\Kfin}$ of $K$-finite vectors in $\pi$ has a natural  structure of $(\fg,K)$-module, and is equipped with the relative topology of $\pi$. 
Consider the complex
\[\cdots\to(\wedge^{l+1}(\fg/\fk)\otimes\pi_{\Kfin})_K
\xrightarrow{\partial_{l+1}}(\wedge^{l}(\fg/\fk)\otimes\pi_{\Kfin})_K
\xrightarrow{\partial_l}(\wedge^{l-1}(\fg/\fk)\otimes\pi_{\Kfin})_K\to\cdots\]
which gives the relative Lie algebra homology groups $H_{\ast}(\fg,K;\pi_{\Kfin})$. 
We equip 
$(\wedge^{l}(\fg/\fk)\otimes\pi_{\Kfin})_K$ with the quotient topology of $\wedge^{l}(\fg/\fk)\otimes\pi_{\Kfin}$, 
$\Ker (\partial_l)$ with the relative topology of $(\wedge^{l}(\fg/\fk)\otimes\pi_{\Kfin})_K$, and 
$H_l(\fg,K;\pi_{\Kfin})$ with the quotient topology of $\Ker(\partial_l)$.

\begin{lemma}[{\cite[Theorem 7.7]{CS21}}]\label{lem:comparison}
There exists a natural isomorphism
$H_*(G,\pi)\cong H_*(\fg,K;\pi_{\Kfin})$ of linear topological spaces. 
\end{lemma}

\begin{remark}
More precisely, it is proved in \cite[Theorem 7.7]{CS21} that 
there exists a natural isomorphism between $H_*(G,\pi)$ and the homology of the complex 
\begin{align}\label{eq:standard}
\cdots\to(\wedge^{l+1}(\fg/\fk)\otimes\pi)_K
\to(\wedge^{l}(\fg/\fk)\otimes\pi)_K
\to(\wedge^{l-1}(\fg/\fk)\otimes\pi)_K\to\cdots
\end{align}
as linear topological spaces. 
By Lemma \ref{lem:inv=coinv}, we have
\[(\wedge^{l}(\fg/\fk)\otimes\pi)_K
=((\wedge^{l}(\fg/\fk)\otimes\pi)_{\Kfin})_K
=(\wedge^{l}(\fg/\fk)\otimes\pi_{\Kfin})_K\] and the assertion of Lemma \ref{lem:comparison} follows. 
\end{remark}

\begin{lemma}\label{lem:inv=coinv}
For any SF representation $\Pi$ of $K$, the linear topological spaces $\Pi_K$, $(\Pi_{\Kfin})_K$ are naturally isomorphic to the closed subspace $\Pi^K$ consisting of $K$-invariant elements.
\end{lemma}

Lemma \ref{lem:inv=coinv} is well-known, but we write the proof in Appendix for the convenience of the reader.

Remark that $\pi\mapsto\pi_{\Kfin}$ is an exact functor from the category of SF representations of $G$ to the category of $(\fg,K)$-modules (without topology). From long exact sequences for relative Lie algebra homology, Lemma \ref{lem:comparison} gives the following
\begin{lemma}[{\cite[Corollary 7.8]{CS21}}]\label{lem:long exact}
For an exact sequence (which is not necessarily strongly exact)
\begin{align*}
0\to\pi_1\to\pi_2\to\pi_3\to 0
\end{align*}
of SF representations of $G$, 
we have the following long exact sequence of vector spaces: 
\begin{align*}
\cdots\to H_l(G,\pi_1)\to H_l(G,\pi_2)\to H_l(G,\pi_3)\to H_{l-1}(G,\pi_1)\to\cdots,
\end{align*}
where $H_l(G,\pi_1)\to H_l(G,\pi_2)$ and $H_l(G,\pi_2)\to H_l(G,\pi_3)$ are continuous. 
\end{lemma}

Let $M$ be a $G$-Nash manifold, $Z$ a closed $G$-Nash submanifold, and $E$ a tempered Fr\'{e}chet $G$-vector bundle over $M$.
Put $U:=M\setminus Z$.
Then we obtain an SF representation of $G$ on the space $\cS(M,E)$ of Schwartz sections of $E$, and a subrepresentation on the space $\cS(U,E|_U)$ of Schwartz sections on $U$.
We define $\cS_Z(M,E):=\cS(M,E)/\cS(U,E|_U)$ and its subrepresentation
\[
\cS_Z(M,E)_k:=\Set{f\in\cS(M,E)| \begin{array}{l}\text{$Df=0$ on $Z$ for any differential}\\
\text{operator $D$ on $M$ of order $\le k-1$}\end{array}}/\cS(U,E|_U).
\]
for $k\in\N$.
Write $N_Z^{\vee}(M)$ for the complexification of the conormal bundle for $Z\subset M$.
\begin{lemma}[{\cite[Propositions 8.2 and 8.3]{CS21}}]\label{lem:Borel}
The canonical homomorphisms
\begin{align*}
\cS_Z(M,E)&\to\varprojlim_k \cS_Z(M,E)/\cS_Z(M,E)_k,\\
\cS_Z(M,E)_k/\cS_Z(M,E)_{k+1}&\to\cS(Z,\Sym^k(N_Z^{\vee}(M))\otimes E)\quad\quad(k\in\N)
\end{align*}
are isomorphisms of SF representations of $G$.
\end{lemma}

Our aim is to give an upper bound for the dimension of coinvariants of a principal series representation with respect to a symmetric subgroup.
In order to apply the following lemma, we will prove the finite-dimensionality of higher homology groups.
\begin{lemma}[{\cite[Lemma 8.4]{CS21}}]\label{lem:commute}
Let $l\in\N$.
Assume that $H_{l+1}(G,\cS_Z(M,E)/\cS_Z(M,E)_k)$ is finite-dimensional for any $k\in\N$. Then the canonical map 
\[
H_l(G,\cS_Z(M,E))\to\varprojlim_k H_l(G,\cS_Z(M,E)/\cS_Z(M,E)_k)
\]
is a linear isomorphism.
\end{lemma}

Let $\delta_G\colon G\to \R_{>0}$ be the modular function of $G$, which is given by $\delta_G(g)=|\det\Ad_{\fg}(g)|$. 
Remark that $\delta_G$ is an SF representation of $G$. The next lemma describes the Schwartz homology of induced representations.

\begin{lemma}[{Shapiro's lemma, \cite[Proposition 6.7]{CS21}}]\label{lem:Shapiro}
Let $l\in\N$.
Let $H$ be a closed Nash subgroup of $G$ and $\gamma$ an SF representation of $H$. Then there exists a natural isomorphism \[H_l(G,\cS(G/H,G\times_H\gamma))\cong H_l(H,\gamma\cdot\delta_G\cdot\delta^{-1}_H)\] as linear topological spaces.
\end{lemma}

The following lemma relates the zeroth homology and the distinguishedness of an SF representation.
\begin{lemma}[{\cite[Proposition 1.9]{CS21}}]\label{lem:Hausdorff}
Let $\pi$ be an SF representation of $G$ and $l\in\N$.
If $H_l(G,\pi)$ is finite-dimensional, then it is Hausdorff.
In particular, if $H_0(G,\pi)$ is finite-dimensional, then we have a natural identification $H_0(G,\pi)^{\vee}\cong\Hom_G(\pi,\triv)$.
\end{lemma}

The next lemma gives the zeroth homology of tensor products of two irreducible SAF representations by Lemma \ref{lem:Hausdorff} and Proposition \ref{prop:finite}.
\begin{lemma}\label{lem:contragredient}
Let $G$ be a real reductive group, $\pi_1,\pi_2$ irreducible SAF representations of $G$, and $\chi$ a character of $G$.
Then $\Hom_G(\pi_1\hat{\otimes}\pi_2,\chi)$ is one-dimensional if $\pi_1\cong \pi_2^{\vee}\cdot\chi$, and zero otherwise.
Here $\hat{\otimes}$ stands for the completed tensor product of two nuclear Fr\'echet spaces.
\end{lemma}
\begin{proof}
Assume that $f$ is a nonzero element in $\Hom_G(\pi_1\hat{\otimes}\pi_2,\chi)$.
Then $f$ restricts to a nonzero $(\fg_{\C},K)$-homomorphism from $(\pi_1)_{\text{$K$-fin}}\otimes(\pi_2)_{\text{$K$-fin}}$ to $\chi$, which induces a $(\fg_{\C},K)$-homomorphism from $(\pi_1)_{\text{$K$-fin}}$ into $(\pi_2)^{\vee}_{\text{$K$-fin}}\cdot\chi$, where $(\pi_2)^{\vee}_{\text{$K$-fin}}$ denotes the contragredient $(\fg_{\C},K)$-module of $(\pi_2)_{\text{$K$-fin}}$. Since it is nonzero, it is an isomorphism. Then the uniqueness of the Casselman-Wallach globalization proves $\pi_1\cong \pi_2^{\vee}\cdot\chi$, which proves the latter assertion.

We next assume $\pi_1\cong \pi_2^{\vee}\cdot\chi$. By Casselman's subrepresentation theorem, $\pi_1$ can be imbedded into a normalized parabolic induction $\Ind^G_P(\gamma)$, where $P$ denotes a minimal parabolic subgroup of $G$ and $\gamma$ denotes a finite-dimensional irreducible SF representation of $P$.
From $(\pi_1)_{\text{$K$-fin}}\cong(\pi_2)^{\vee}_{\text{$K$-fin}}\cdot\chi$, we may regard $\pi_2$ as a quotient of the SAF representation $\Ind^G_P(\gamma^{\vee})\cdot\chi\cong(\Ind^G_P(\gamma))^{\vee}\cdot\chi$. Then the canonical nonzero $G$-homomorphism from $\Ind^G_P(\gamma)\hat{\otimes}\Ind^G_P(\gamma^{\vee})\cdot\chi$ to $\chi$ induces the one from $\pi_1\hat{\otimes}\pi_2$ to $\chi$.
Since nonzero $(\fg_{\C},K)$-homomorphisms from $(\pi_1)_{\text{$K$-fin}}$ to $(\pi_2)^{\vee}_{\text{$K$-fin}}\cdot\chi$ are unique up to a scalar multiple and the subspace $(\pi_1)_{\text{$K$-fin}}\otimes(\pi_2)_{\text{$K$-fin}}$ of $\pi_1\hat{\otimes}\pi_2$ is dense, we see $\Hom_G(\pi_1\hat{\otimes}\pi_2,\chi)$ is one-dimensional, which completes the proof.
\end{proof}

\subsection{Lie algebra homology}\label{sec:3-2}
In this section, we prepare a criterion (Lemma \ref{lem:spectral seq}) for the vanishing of relative Lie algebra homologies, which will be used in Section \ref{sec:5}.

Let $G$ be a real reductive group and $K$ a maximal compact subgroup. Write $\theta$ for the Cartan involution on $\fg$ with $\fk=\fg^{\theta}$.
Fix a maximally abelian subspace $\fs$ of $\fg^{-\theta}$ and a positive system $\Sigma^+$ of the root system $\Sigma$ for $(\fg,\fs)$.
Let $P=LN$ be a standard parabolic subgroup of $G$ and its Levi decomposition.

Let $Z(\fg_{\C})$ be the center of the enveloping algebra $U(\fg_{\C})$.
A $U(\fg_{\C})$-module $V$ is called $Z(\fg_{\C})$-finite if $V$ is annihilated by an ideal of finite codimension of $Z(\fg_{\C})$.

For any $(\fg_{\C},K)$-module $V$, we have a natural isomorphism
\[
H_{\ast}(\fg_{\C},K;V)^{\vee}\cong H^{\ast}(\fg_{\C},K;V^{\vee})
\]
of vector spaces, and a natural isomorphism
\[
H_{\ast}(\fn_{\C},V)^{\vee}\cong H^{\ast}(\fn_{\C},V^{\vee})
\]
of $(\fl_{\C},K\cap L)$-modules
by \cite[Theorem 3.1]{KV95}, where the symbols $^\vee$ of $H_{\ast}(\fg_{\C},K;V)^{\vee}, V^{\vee}, H_{\ast}(\fn_{\C},V)^{\vee}$ denote the contragredient, $K$-finite contragredient, $(L\cap K)$-finite contragredient, respectively.
Hence we can translate results on relative Lie algebra cohomology (resp.\,$\fn_{\C}$-cohomology) into those on relative Lie algebra homology (resp.\,$\fn_{\C}$-homology) as follows.

\begin{lemma}[{\cite[Corollary I.4.2 and Section I.5.1]{BW00}}]\label{lem:vanishing}
Let $V_1, V_2$ be $Z(\fg_{\C})$-finite $(\fg_{\C},K)$-modules.
Assume that $V_1$ is finite-dimensional, and that any infinitesimal character of the composition factors of $V_2$ as $(\fg_{\C},K)$-modules does not appear in the infinitesimal characters of composition factors of the contragredient $(\fg_{\C},K)$-module $V_1^{\vee}$ of $V_1$. Then $H_{\ast}(\fg_{\C},K;V_1\otimes V_2)=0$.
\end{lemma}

\begin{lemma}[{\cite[Corollary3.1.6]{Vog81}}]\label{lem:n-homology2}
Let $l\in\N$ and $V$ be a $Z(\fg_{\C})$-finite $\fg_{\C}$-module. 
Then the $\fl_{\C}$-module $H_{l}(\fn_{\C},V)$ is $Z(\fl_{\C})$-finite.
\end{lemma}

\begin{lemma}\label{lem:spectral seq}
Let $l\in\N$ and $L'$ be a closed subgroup of $L$.
Write $P'=L'N$.
Let $\fv$ be a finite-dimensional real Lie algebra where $P'$ acts as Lie algebra automorphisms. 
Let $V$ be a $Z(\fg_{\C})$-finite $(\fg_{\C},K)$-module and $W$ a finite-dimensional $((\fp'\ltimes\fv)_{\C},K\cap L')$-module.
Assume that 
the action of $\fn$ on $\fv$ and $W$ are nilpotent, and that 
there exists $X\in\fl'\cap\fz(\fl)$ such that for any $p,q\in\N$ with $p+q\le l$, the eigenvalues of $X$ on the $\fl'_{\C}$-module $H_p(\fn_{\C},V)\otimes\wedge^q\fv_{\C}\otimes W$ are nonzero.
Then 
\[H_l((\fp'\ltimes\fv)_{\C},K\cap L';V\otimes W)=0.\]
Here we regard $V$ as a $((\fp'\ltimes\fv)_{\C},K\cap L')$-module by letting $\fv$ act trivially.
\end{lemma}
\begin{proof}
By the Hochschild-Serre spectral sequence for relative Lie algebra homology, we have
\begin{align*}
H_l((\fp'\ltimes\fv)_{\C},K\cap L';V\otimes W)
\cong\bigoplus_{0\le q\le l}H_{l-q}(\fp'_{\C},K\cap L';V\otimes H_q(\fv_{\C},W)).
\end{align*}
The $(\fp'_{\C},K\cap L')$-module $H_q(\fv_{\C},W)$ can be written as a subquotient of $\wedge^q\fv_{\C}\otimes W$ by the standard complex.
By long exact sequences, it suffices to prove 
\[H_{l-q}(\fp'_{\C},K\cap L';V\otimes Z_1)=0\]
for any $q\in[0,l]$ and any composition factor $Z_1$ of the $(\fp'_{\C},K\cap L')$-module $\wedge^q\fv_{\C}\otimes W$.
Since $\fn$ acts on $\fv$ and $W$ nilpotently, the action of $\fn$ on $Z_1$ is trivial.
Again by the Hochschild-Serre spectral sequence, we see
\begin{align}\label{eq:zero}
H_{l-q}(\fp'_{\C},K\cap L';V\otimes Z_1)\cong\bigoplus_{0\le p\le l-q}H_{l-p-q}(\fl'_{\C},K\cap L';H_p(\fn_{\C},V)\otimes Z_1).
\end{align}
Fix $p\in [0,l-q]$. Since $V$ is $Z(\fg_{\C})$-finite, Lemma \ref{lem:n-homology2} shows that the $(\fl_{\C},L\cap K)$-module $H_p(\fn_{\C},V)$ admits a finite filtration whose graded pieces $Z_2$ have infinitesimal characters.
Since $\fz(\fl)$ acts on $Z_1$ by a character, the element $X$ in $\fl'\cap\fz(\fl)$ acts on $Z_2\otimes Z_1$ by a scalar $C$. By the assumption, we have $C\neq 0$ and obtain 
$H_{l-p-q}(\fl'_{\C},K\cap L';Z_2\otimes Z_1)=0$ by Lemma \ref{lem:vanishing}.
By long exact sequences, we see that \eqref{eq:zero} is zero, and the proof is complete.
\end{proof}

\section{Preliminaries on symmetric subgroups}\label{sec:4}
In this section, we prepare technical lemmas which will play important roles in Section \ref{sec:5}.

Let $G$ be a real reductive group, $H$ a symmetric subgroup of $G$, that is, there exists an involution $\sigma$ on $G$ with $G^{\sigma}_0\subset H\subset G^{\sigma}$.
Here $G^{\sigma}_0$ denotes the neutral component of $G^{\sigma}$. 

We take a Cartan involution $\theta$ on $\fg$ which commutes with $\sigma$, and a $\sigma$-stable maximally abelian subspace $\fs$ of $\fg^{-\theta}$.
Write $\Sigma=\Sigma(\fg,\fs)$ for the root system for $(\fg,\fs)$, and fix a positive system $\Sigma^+$ of $\Sigma$.
Write $N_{\min}$ for the unipotent radical of the minimal parabolic subgroup $P_{\min}$ corresponding to $\Sigma^+$.
Let $P=MAN$ be a standard parabolic subgroup and its Langlands decomposition. Put $L=MA$.

\begin{lemma}\label{lem:factor}
The map
\begin{align}\label{eq:factor}
(H\cap L)\ltimes (L\cap\sigma N)\to (H\cap P)/(H\cap N);\quad (l,\sigma n)\mapsto l\cdot\exp\circ(\id+\sigma)\circ\log(\sigma n)(H\cap N)
\end{align}
is an isomorphism of Nash groups.
Here $\log$ denotes the inverse of the exponential map from $\sigma\fn$ to $\sigma N$.
\end{lemma}
\begin{proof}
We recall that the exponential map for a nilpotent subalgebra of $\fg$ and the logarithm map for a closed unipotent Nash subgroup of $G$ are Nash.
Hence the map \eqref{eq:factor} is Nash.

Consider a closed unipotent Nash subgroup $N'_{\min}:=(L\cap\sigma N_{\min})N$ of $G$.
Since $\fu:=(\fp\cap\sigma\fn)+(\fn\cap\sigma\fp)$ is a $\sigma$-stable subalgebra of $\fn'_{\min}$, the image $U:=\exp(\fu)$ of the exponential map is a closed normal $\sigma$-stable unipotent Nash subgroup of $P\cap\sigma P$.
Hence $P\cap\sigma P$ has a subgroup $(L\cap \sigma L)\ltimes U$.

Let us prove
\begin{align}\label{eq:factor2}
P\cap\sigma P=(L\cap \sigma L)\ltimes U.
\end{align}
It suffices to show $P\cap\sigma P\subset(L\cap \sigma L)\ltimes U$. Let $p\in P\cap\sigma P$. Then $p=ln$ for some $l\in L, n\in N$.
By multiplying an element in $\exp(\fn\cap\sigma\fp)$ from the right, we may assume $n\in N\cap\sigma\bar{N}$. By $\Ad(n^{-1})\fa=\Ad(p^{-1}l)\fa=\Ad(p^{-1})\fa\subset \sigma\fp$, we see $a^{-1}n^{-1}an\in \sigma P\cap\sigma\bar{N}=\{1\}$ for any $a\in A$. Hence $n=1$ and $p=l\in L\cap\sigma P$.

Write $p=\sigma(l')\sigma(n')$ for some $l'\in L, n'\in N$. 
Then from $p\in L$ we see that $\Ad(\sigma(n'))\fa=\Ad(\sigma(l')^{-1}p)\fa=\Ad(\sigma(l')^{-1})\fa\subset\sigma\fl$. Hence $a\sigma(n')a^{-1}\sigma(n')^{-1}\in\sigma L\cap\sigma N=\{1\}$ for any $a\in A$, and $\sigma(n')$ belongs to $L\cap\sigma(N)=\exp(\fl\cap\sigma\fn)\subset U$. Therefore $\sigma(l')=p\sigma(n')^{-1}$ belongs to $L\cap\sigma L$ and $p=\sigma(l')\sigma(n')$ belongs to $(L\cap\sigma L)\ltimes U$, which proves \eqref{eq:factor2}.

Since the factorization \eqref{eq:factor2} is $\sigma$-stable, we obtain
\begin{align}\label{eq:factor3}
H\cap P=(H\cap L)\ltimes U^{\sigma}
\end{align}
by taking intersections with $H$.
The Lie algebra of $U^{\sigma}$ equals $\fu^{\sigma}=(\id+\sigma)(\fl\cap\sigma\fn)+(\fh\cap\fn)$ and has an ideal $\fh\cap\fn$. By $[(\id+\sigma)X,(\id+\sigma)Y]\in(\id+\sigma)[X,Y]+(\fh\cap\fn)$ for any $X,Y\in\fl\cap\sigma\fn$, the linear map
\[
\fl\cap\sigma\fn\to\fu^{\sigma}/(\fh\cap\fn);X\mapsto (\id+\sigma)X+(\fh\cap\fn)
\]
is an isomorphism of nilpotent Lie algebras.
It induces an isomorphism 
\begin{align}\label{eq:composite}
L\cap\sigma N\xrightarrow{\cong} U^{\sigma}/(H\cap N);\sigma n\mapsto \exp\circ(\id+\sigma)\circ\log(\sigma n)\cdot (H\cap N)
\end{align}
of $1$-connected nilpotent Lie groups, which is Nash as discussed at the beginning of the proof.
The assertion now follows from \eqref{eq:factor3} and \eqref{eq:composite}.
\end{proof}

For $\alpha\in\Sigma$, we write $\fg_{\alpha}$ for the root space of root $\alpha$.

Let us recall the definition of the associated (resp. the dual) symmetric pair of $(\fg,\fh)$. 
\emph{The associated symmetric pair} of $(\fg,\fh)$ is the symmetric pair $(\fg^a,\fh^a):=(\fg,\fg^{\theta\sigma})$. Recall that $\fh^a$ is stable under the Cartan involution $\theta$ of $\fg^a$.
Write $\fg_{\C}$ for the complexification of $\fg$ and extend $\theta$ and $\sigma$ to $\fg_{\C}$ so that they are complex linear. 
Then $\fg^d:=\fg^{\sigma\theta}+\sqrt{-1}\fg^{-\sigma\theta}$ is a real form of $\fg_{\C}$, and $\fh^d:=(\fg^d)^{\theta}$ is stable under the Cartan involution $\sigma$ of $\fg^d$. 
the symmetric pair $(\fg^d,\fh^d)$ is called \emph{the dual symmetric pair} of $(\fg,\fh)$. 

\begin{lemma}\label{lem:extension}
Assume that $(\fg,\fh)$ is one of the following symmetric pairs:
\begin{enumerate}
\item
$(\fg'\oplus\fg',\fg')$ where $\fg'$ denotes a real simple Lie algebra,
\item
$(\mathfrak{sl}_{2n}(\R),\mathfrak{sl}_n(\C)+\sqrt{-1}\R)$, or
\item
$(\fe_{6(6)},\mathfrak{sl}_3(\H)\oplus\mathfrak{su}(2))$.
\end{enumerate}
Then $\Sigma^{\sigma}$ is empty and $\fg_{\alpha}\subset\fg^{\theta\sigma}$ for any $\alpha\in\Sigma^{-\sigma}$ (see Section \ref{sec:intro-3} for the definition of $\Sigma^{\sigma}, \Sigma^{-\sigma}$).
\end{lemma}
\begin{proof}
It suffices to show 
\begin{align}\label{eq:centralizer}
\fz_{\fg}(\fs^{-\sigma})=\fz_{\fg}(\fs),\quad\fz_{\fg}(\fs^{\sigma})\subset \fh^a+\fz_{\fg}(\fs).
\end{align}

By the assumption on $(\fg,\fh)$, the description of $(\fg^a,\fh^a)$ \cite[Lemma (1.13.1) and (1.16)]{OS84} and data of the restricted root systems of $(\fg,\fh)$ and $(\fg^a,\fh^a)$ \cite[(6.8) and Table V]{OS84}, the sum of the split rank of $(\fg,\fh)$ and that of $(\fg^a,\fh^a)$ equals the real rank of $\fg$.
Hence maximally abelian $\sigma$-stable subspaces of $\fg^{-\theta}$ are conjugate under the neutral component of $K\cap H$ by \cite[Lemma 2.4]{OS84}.
Hence we may take a particular choice of $\fs$.

We first consider the case $(\fg,\fh)=(\fg'\oplus\fg',\fg')$, where $\fg'$ denotes a real simple Lie algebra.
In this case, $\fs=\fs'\oplus\fs'$ for some maximally abelian subspace $\fs'$ of $(\fg')^{-\theta}$.
Then we see that $\fz_{\fg}(\fs^{\sigma})$ and $\fz_{\fg}(\fs^{-\sigma})$ are equal to $\fz_{\fg}(\fs)$, and \eqref{eq:centralizer} follows.

We next consider the case $(\fg,\fh)=(\sl_{2n}(\R),\sl_n(\C)+\sqrt{-1}\R)$.
Take a Cartan involution $\theta$ on $\fg$ given by $\theta(Y)=-{}^t{Y}$ for $Y\in \fg$.
Fix a complex structure $J=\begin{pmatrix}0&-J_n\\J_n&0\end{pmatrix}$ on $\R^{2n}$, where $J_n$ denotes the anti-diagonal $n$-by-$n$ matrix with anti-diagonal entries one. Write $\sl_n(\C)$ as the subalgebra of $\sl_{2n}(\R)$ commuting with the complex structure.
We choose $\fs$ to be the subalgebra consisting of diagonal matrices with trace zero. Then by \[\fs^{-\sigma}=\Set{\diag(h_1,\ldots,h_n,-h_n,\ldots,-h_1)|\sum_{1\le i\le n} h_i=0},\] we see $\fz_{\fg}(\fs^{-\sigma})=\fs=\fz_{\fg}(\fs)$.
Moreover, we have $\fz_{\fg}(\fs^{\sigma})=\{\text{anti-diagonal matrices}\}+\fs\subset\fh^a+\fs$ since anti-diagonal matrices belong to $\fh^a$.

We finally consider the case $(\fg,\fh)=(\fe_{6(6)},\mathfrak{sl}_3(\H)\oplus\mathfrak{su}(2))$.
In this case, we see $\fh^a=\ff_{4(4)}$ by \cite[(1.16)]{OS84}.
Take $\fs_{\C}$ as in \cite[p. 78]{Yok09}. Then the description of roots \cite[Theorem 3.6.4]{Yok09} implies \eqref{eq:centralizer}, which completes the proof.
\end{proof}

For a closed subgroup $R$ of $G$, define 
\[
\Sigma_R:=\Set{\alpha\in\Sigma|\fg_{\alpha}\subset\fr}.
\]
Since $\fl\cap\sigma\fn$ is nilpotent, the modular function $\delta_{L\cap\sigma N}$ of $L\cap\sigma N$ is trivial.
Via the isomorphism \eqref{eq:factor}, we extend the domain of $\delta_{L\cap\sigma N}$ to $H\cap P$ and regard it as the modular function $\delta_{(H\cap P)/(H\cap L)}$.
\begin{lemma}\label{lem:rho}
Assume that any simple component of the symmetric pair $(\fg,\fh)$ is written as 
\begin{enumerate}
\item\label{ind:rho-first} 
$(\fg'\oplus\fg',\fg')$ or $(\fg'_{\C},\fg')$ for some real simple Lie algebra $\fg'$,
\item\label{ind:rho-second} 
$(\mathfrak{sl}_{2n}(\R),\mathfrak{sl}_n(\C)+\sqrt{-1}\R)$ or $(\mathfrak{sl}_n(\H),\mathfrak{sl}_n(\C)+\sqrt{-1}\R)$, or
\item\label{ind:rho-third} 
$(\fe_{6(6)},\mathfrak{sl}_3(\H)\oplus\mathfrak{su}(2))$ or $(\fe_{6(-26)},\mathfrak{sl}_3(\H)\oplus\mathfrak{su}(2))$.
\end{enumerate}

Then $\delta_P^{1/2}\cdot\delta_{H\cap P}^{-1}=\delta_{L\cap\sigma N}^{-1/2}$ as characters of $H\cap P$, and we have $\fs^{\sigma}\cap[\fg,\fg]=\fs^{\sigma}\cap[\fh,\fh]$.
\end{lemma}
\begin{proof}
From \eqref{eq:factor3}, we have $\fh\cap\fp=(\fh\cap\fl)\ltimes\fu^{\sigma}$.
The adjoint actions of the subalgebra $\fu^{\sigma}$ on $\fp,\fh\cap\fp$ and $\fu^{\sigma}/(\fh\cap\fn)\cong\fl\cap\sigma\fn$ are all nilpotent.
Since $\fh\cap\fl$ is reductive, we see $\fh\cap\fl=\fs^{\sigma}+[\fh\cap\fl,\fh\cap\fl]$.
From these observations, it suffices to prove 
$\rho_{\fp}-2\rho_{\fh\cap\fp}=-\rho_{\fl\cap\sigma\fn}$ as linear functionals of $\fs^{\sigma}$ (see Section \ref{sec:intro-3} for the definition of $\rho_{\ast}$).
Set 
\[\Sigma_{1}:=\Sigma_P\cap\Sigma^{\sigma},\quad
\Sigma_{2}:=(\Sigma_P\cap\sigma\Sigma_P)\setminus\Sigma_1,\quad
\Sigma_{3}:=\Sigma_L\cap\sigma\Sigma_{\bar{N}},\quad
\Sigma_{4}:=\Sigma_N\cap\sigma\Sigma_{\bar{N}}.
\]
We see $\Sigma_P=\bigsqcup_{1\le i\le 4}\Sigma_{i}$ and
\begin{align*}
\rho_{\fp}&=\frac{1}{2}\sum_{1\le i\le 4}\sum_{\alpha\in\Sigma_i}\dim\fg_{\alpha}\cdot\alpha|_{\fs^{\sigma}},\quad\quad
\rho_{\fl\cap\sigma\fn}=-\frac{1}{2}\sum_{\alpha\in\Sigma_3}\dim\fg_{\alpha}\cdot\alpha|_{\fs^{\sigma}},\\
\rho_{\fh\cap\fp}&=\frac{1}{2}\sum_{\alpha\in\Sigma_1}\dim \fg_{\alpha}^{\sigma_j}\cdot\alpha|_{\fs^{\sigma}}+\frac{1}{2}\sum_{\alpha\in\Sigma_2}\frac{\dim\fg_{\alpha}}{2}\cdot\alpha|_{\fs^{\sigma}}.
\end{align*}
Since $\Sigma_4$ is $(-\sigma)$-stable, we see $\sum_{\alpha\in\Sigma_4}\dim\fg_{\alpha}\cdot\alpha|_{\fs^{\sigma}}=0$.
Hence we have
\[
\rho_{\fp}-2\rho_{\fh\cap\fp}+\rho_{\fl\cap\sigma\fn}
=\frac{1}{2}\sum_{\alpha\in\Sigma_{1}}(\dim\fg_{\alpha}^{-\sigma}-\dim\fg_{\alpha}^{\sigma})\cdot\alpha|_{\fs^{\sigma}}.
\]
Therefore the first half of the assertion reduces to proving $\dim\fg_{\alpha}^{\sigma}=\dim\fg_{\alpha}^{-\sigma}$ for $\alpha\in\Sigma^{\sigma}$.
It also reduces to proving $\dim\fg_{\alpha}^{\sigma}=\dim\fg_{\alpha}^{-\sigma}$ for any nonzero $\fs^{\sigma}$-weights in $\fg$.

We take a maximally abelian subspace $\fj$ of $\fh\cap\fg^{-\theta}$ containing $\fs^{\sigma}$.
Then the set $\Sigma(\fg,\fj)$ of nonzero $\fj$-weights in $\fg$ is naturally identified with the restricted root system for $(\fg^a,\fh^a)$ (cf. \cite[Theorem 2.11]{OS84}).
Let us prove the following stronger assertion: $\dim\fg_{\alpha}^{\sigma}=\dim\fg_{\alpha}^{-\sigma}$ for $\alpha\in\Sigma(\fg,\fj)$.

We may assume that $\fg$ is simple. 
Recall the definitions of the associated $(\fg^a,\fh^a)$ and the dual $(\fg^d,\fh^d)$ symmetric pair of a symmetric pair $(\fg,\fh)$ before Lemma \ref{lem:extension}.
For each case \ref{ind:rho-first}--\ref{ind:rho-third}, a symmetric pair is written as the associated $(\fg^{ada},\fh^{ada})=(\fg^{ada},\fh)$ of the dual of the associated of the other symmetric pair $(\fg,\fh)$.
Since the dimension of $\fg_{\alpha}^{\pm\sigma}$ for $\alpha\in\Sigma(\fg,\fj)$ does not change if we replace $(\fg,\fh)$ by $(\fg^{ada},\fh^{ada})$, we only need to consider the case $(\fg,\fh)=(\fg'\oplus\fg',\fg'), (\mathfrak{sl}_{2n}(\R),\mathfrak{sl}_n(\C)+\sqrt{-1}\R)$, or $(\fe_{6(6)},\mathfrak{sl}_3(\H)\oplus\mathfrak{su}(2))$.
In this case, the proof of Lemma \ref{lem:extension} shows $\fj=\fs^{\sigma}$, and $\Sigma^{\sigma}=\emptyset$ implies $\dim\fg_{\alpha}^{\sigma}=\dim\fg_{\alpha}^{-\sigma}$ for any $\alpha\in\Sigma(\fg,\fs^{\sigma})$.

Moreover, $\dim\fg_{\alpha}^{\sigma}=\dim\fg_{\alpha}^{-\sigma}$ for $\alpha\in\Sigma(\fg,\fj)$ shows $\Sigma(\fg,\fj)=\Sigma(\fh,\fj)$. Hence $\fj\cap[\fg,\fg]=\fj\cap[\fh,\fh]$ and the second half of the assertion follows.
\end{proof}

Write $X_{\alpha}$ for the elements in $\fs\cap[\fg,\fg]$ corresponding to $\alpha\in\Sigma$ under the identification $\fs\cap[\fg,\fg]\cong (\fs\cap[\fg,\fg])^{\vee}$ given by the Killing form $(\cdot,\cdot)$ of $[\fg,\fg]$.
\begin{lemma}\label{lem:positive}
The $\fs^{\sigma}$-module $\fg$ decomposes into $(\fh+\fp)\oplus (\bar{\fn}\cap\sigma\bar{\fn})^{-\sigma}$.
Moreover, if $\fg\neq\fh+\fp$, then there exists $X\in \fs^{\sigma}\cap[\fg,\fg]$ with the following properties:
\begin{enumerate}
\item\label{ind:first}
$X$ belongs to $\sum_{\alpha\in\Xi}\R_{>0}X_{\alpha}$, where $\Xi:=(\Sigma_N\cap\sigma\Sigma_L)\cup(\sigma\Sigma_N\cap\Sigma_L)\cup(\Sigma_N\cap\sigma\Sigma_N)$.
\item\label{ind:second}
$\alpha(X)>0$ for any $\alpha\in\Xi$.
\item\label{ind:third}
$X$ commutes with $\fl\cap\sigma\fl$.
\end{enumerate}
\end{lemma}
\begin{proof}
Let us prove the first statement.
By $\Sigma=\Sigma_P\sqcup(\Sigma_{\bar{N}}\cap\sigma\Sigma_P)\sqcup(\Sigma_{\bar{N}}\cap\sigma\Sigma_{\bar{N}})$ and $\fg_{\alpha}\subset (\id+\sigma)\fg+\sigma\fg_{\alpha}\subset\fh+\fp$ for any $\alpha\in\Sigma_{\bar{N}}\cap\sigma\Sigma_P$, we have $\fg=\fh+\fp+(\bar{\fn}\cap\sigma\bar{\fn})$.
By $(\fh+\fp)\cap\bar{\fn}\cap\sigma\bar{\fn}=(\bar{\fn}\cap\sigma\bar{\fn})^{\sigma}$, we obtain $\fg=(\fh+\fp)\oplus(\bar{\fn}\cap\sigma\bar{\fn})^{-\sigma}$.

We next prove the second statement.
Let $W_{L\cap\sigma L}$ be the Weyl group for $(\fl\cap\sigma\fl,\fs)$.
We remark that the groups $W_{L\cap\sigma L}$ and $\{\id,\sigma\}$ act on the set $\Xi$. 
Write $\fS_{\Xi}$ for the group of permutations of $\Xi$ generated by $W_{L\cap\sigma L}$ and $\sigma$.
By the assumption $\fg\neq\fh+\fp$, the subset $\Xi$ of $\Sigma^+$ is nonempty.
Hence the convex cones
\[
E:=\R_{\ge 0}\Set{\sum_{\tau\in\fS_{\Xi}}X_{\tau\alpha}\in\fs^{\sigma}\cap[\fg,\fg]|\alpha\in\Xi}\subset V:=\R E
\]
satisfy $E\neq 0, V$ and $E\cap (-E)=0$.
Any element $X$ in $E$ satisfies \ref{ind:third} since $X$ is invariant under the reflections with respect to roots in $\Sigma_L\cap\sigma\Sigma_L$.
Since the dual cone $E^{\ast}$ of $E$ in $V$ is written as 
\[
E^{\ast}=\Set{X\in V|\alpha(X)\ge 0\text{ for any $\alpha\in\Xi$}},
\]
it suffices to show that the closed convex cone $E\cap E^{\ast}$ has an inner point in $V$, which is equivalent to $(E+E^{\ast})\cap -(E+E^{\ast})=0$ by the general theory of convex cones (cf. \cite[Theorem 1.2.5]{Fen53}).
Let $u\in (E+E^{\ast})\cap -(E+E^{\ast})$ and write $u=v_1+w_1=-v_2-w_2$ for some $v_1,v_2\in E$ and $w_1,w_2\in E^{\ast}$. We see $0=\|v_1+w_1+v_2+w_2\|^2=\|v_1+v_2\|^2+\|w_1+w_2\|^2+2(v_1+v_2,w_1+w_2)$ and the three terms on the right hand side are nonnegative. Hence $v_1=-v_2, w_1=-w_2$. From $E\cap (-E)=0$ and $E^{\ast}\cap(-E^{\ast})=0$, we obtain $u=0$. Hence $(E+E^{\ast})\cap -(E+E^{\ast})=0$.
\end{proof}

\section{Homology of principal series representations and orbit decompositions}\label{sec:5}

\subsection{Decomposition into $H$-orbits}\label{sec:5-1}

Retain notation in Section \ref{sec:3-2}.
Let $H$ be a closed (not necessarily symmetric) Nash subgroup of $G$. 
Assume that the set of double cosets $\fX:= H\backslash G/P$ is finite.

Let $\gamma$ be an SF representation of $L$.
By letting $N$ act trivially, we regard $\gamma$ as a representation of $P$.
Consider the normalized induction $\Pi=\Ind^G_P(\gamma)$ of $\gamma$, which is given by the space $\cS(G/P,\cV)$ of Schwartz sections of the homogeneous tempered bundle $\cV:=G\times_P(\gamma\cdot\delta_P^{1/2})$.
Let $\chi$ be a finite-dimensional SF representation of $H$.

In this section, we give an upper estimate of the dimension of $\Hom_H(\Pi,\chi^{\vee})$.

Let $\{g_j\}_{j=1}^{j_0}$ be a complete set of representatives of $\fX= H\backslash G/P$.
Equip $\fX$ with the closure relation.
Let us fix an order-preserving bijection from $\fX$ to $[1,j_0]$.
For $j\in[1,j_0]$, we write $\cO_j$ for the preimage of $j$.
Since $U_j:=(\bigsqcup_{j<k\le j_0}\cO_k)/P$ is an $H$-invariant open subset of $G/P$, the space of Schwartz sections
\[\Pi_j:=\cS(U_j,\cV|_{U_j})\]
of the restriction $\cV|_{U_j}$
is an $H$-subrepresentation of $\Pi$.
Hence we obtain a descending sequence
\[\Pi_0\otimes\chi^{\vee}\supset\Pi_1\otimes\chi^{\vee}\supset\Pi_2\otimes\chi^{\vee}\supset\cdots\supset\Pi_{j_0}\otimes\chi^{\vee}=0\]
of SF representations of $H$, where we set $\Pi_0:=\Pi$.

Put $\rho_j:=\Pi_{j-1}/\Pi_j$ for $j\in[1,j_0]$.
Since tensoring with $\chi^{\vee}$ is an exact functor, the short exact sequence $0\to\Pi_{j}\to\Pi_{j-1}\to\rho_j\to 0$ and Lemma \ref{lem:long exact} gives us an long exact sequence
\[\cdots\to H_l(H,\Pi_j\otimes\chi^{\vee})\to H_l(H,\Pi_{j-1}\otimes\chi^{\vee})\to H_l(H,\rho_j\otimes\chi^{\vee})\to H_{l-1}(H,\Pi_j\otimes\chi^{\vee})\to\cdots.\]
In particular, $\dim H_l(H,\Pi_{j-1}\otimes\chi^{\vee})\le \dim H_l(H,\Pi_j\otimes\chi^{\vee})+\dim H_l(H,\rho_j\otimes\chi^{\vee})$ for any $l$. Hence
\begin{align}\label{eq:orbits}
\dim H_l(H,\Pi\otimes\chi^{\vee})\le \sum_{j=1}^{j_0}\dim H_l(H,\rho_j\otimes\chi^{\vee}).
\end{align}

Fix $j\in[1,j_0]$.
Set $Q_j:=H\cap g_jPg_j^{-1}$.
Recall that $\cO_j$ is a closed Nash submanifold of $U_{j-1}$ and isomorphic to $H/Q_j$.
By Lemma \ref{lem:Borel}, $\rho_j$ admits a descending sequence $\{\rho_{j,k}\}_{k\in\N}$ of $H$-subrepresentations satisfying $\rho_{j,0}=\rho_{j}$ and
\begin{align*}
\rho_j\otimes\chi^{\vee}&\cong \varprojlim_k(\rho_j/\rho_{j,k})\otimes\chi^{\vee},\\
(\rho_{j,k}/\rho_{j,k+1})\otimes\chi^{\vee}&\cong\cS(H/Q_j,H\times_{Q_j}({}^{g_j^{-1}}(\gamma\cdot\delta_P^{1/2})\otimes\Sym^k(\fg/(\fh+\Ad(g_j)\fp))^{\vee}_{\C}))\otimes\chi^{\vee}\\
&\cong\cS(H/Q_j,H\times_{Q_j}({}^{g_j^{-1}}(\gamma\cdot\delta_P^{1/2})\otimes\Sym^k(\fg/(\fh+\Ad(g_j)\fp))^{\vee}_{\C}\otimes\chi^{\vee}))\quad(k\in\N)
\end{align*}
as $H$-representations.
For $l,k\in\N$, we define 
\begin{align}\label{df:d}
d_{l,j,k}:=\dim H_l(H,(\rho_{j,k}/\rho_{j,k+1})\otimes\chi^{\vee}).
\end{align}
Put $H^j:=g_j^{-1}Hg_j$.
For a closed subgroup $R$ of $G$, we define $R^j:=H^j\cap R$.
\begin{remark}
When $H$ is the group of fixed points of an involution $\sigma$ on $G$, the group $R^j$ equals the group $(R\cap\sigma_j(R))^{\sigma_j}$ of fixed points in $R$ of the involution $\sigma_j:=\Psi(g_j)^{-1}\circ\sigma\circ\Psi(g_j)$ on $G$.
\end{remark}
By Lemmas \ref{lem:comparison} and \ref{lem:Shapiro}, we obtain the following isomorphisms of vector spaces:
\begin{align}\label{eq:pieces}
&H_l(H,(\rho_{j,k}/\rho_{j,k+1})\otimes\chi^{\vee})\\\notag
&\cong
H_l(Q_j,{}^{g_j^{-1}}(\gamma\cdot\delta_P^{1/2})\otimes\Sym^k(\fg/(\fh+\Ad(g_j)\fp))^{\vee}_{\C}\otimes\chi^{\vee}\cdot\delta_H\cdot\delta_{Q_j}^{-1})\\\notag
&\cong H_l(P^j,\gamma\otimes\Sym^k(\fg/(\fh^j+\fp))^{\vee}_{\C}\otimes{}^{g_j}\chi^{\vee}\cdot\delta_{H^j}\cdot\delta_P^{1/2}\cdot\delta_{P^j}^{-1})\\\notag
&\cong H_l(\fp^j_{\C},K\cap L^j;\gamma_{\text{$(K\cap L^j)$-fin}}\otimes\Sym^k(\fg/(\fh^j+\fp))^{\vee}_{\C}\otimes{}^{g_j}\chi^{\vee}\cdot\delta_{H^j}\cdot\delta_P^{1/2}\cdot\delta_{P^j}^{-1}).
\end{align}

\begin{lemma}\label{lem:bound}
Let $l\in\N$.
If $d_{l+1,j,k}$ is finite for any $j\in[1,j_0]$ and $k\in\N$, then
\[
\dim H_l(H,\Pi\otimes\chi^{\vee})
\le\sum_{j=1}^{j_0}\sum_{k=0}^{\infty} d_{l,j,k}.
\]
In particular, if $\gamma$ is finite-dimensional and $d_{l,j,k}=0$ for sufficiently large $k$, then $H_l(H,\Pi\otimes\chi^{\vee})$ is finite-dimensional.
\end{lemma}
\begin{proof}
By the short exact sequence $0\to\rho_{j,a}/\rho_{j,a+1}\to\rho_j/\rho_{j,a+1}\to\rho_j/\rho_{j,a}\to 0$ for $a\in[0,k-1]$ and Lemma \ref{lem:long exact}, we obtain 
\[
\dim H_l(H,(\rho_j/\rho_{j,k})\otimes\chi^{\vee})\le\sum_{a=0}^{k-1}d_{l,j,a}.
\]
Hence $H_{l+1}(H,(\rho_j/\rho_{j,k})\otimes\chi^{\vee})$ is finite-dimensional by the assumption, and 
\[
\dim H_l(H,\rho_j\otimes\chi^{\vee})=\dim\varprojlim_k H_l(H,(\rho_j/\rho_{j,k})\otimes\chi^{\vee})\le\sum_{k=0}^{\infty}d_{l,j,k}
\]
by Lemma \ref{lem:commute}.
Therefore the assertion follows from the inequality \eqref{eq:orbits}.
\end{proof}

\subsection{Finiteness for symmetric subgroups}\label{sec:5-2}
From this section, we assume that a real reductive group $G$ is of inner type (i.e. $\Ad(G)$ is contained in the group of inner automorphisms of $\fg_{\C}$), that $H$ is a symmetric subgroup of $G$, and that an SF representation $\gamma$ of $L$ has finite length, unless otherwise stated.
Take $\theta, \fs, N_{\min}, P_{\min}$ as in Section \ref{sec:4}.
Retain notation as in Section \ref{sec:5-1}.

In this section, we prove the finite-dimensionality of the Schwartz homology $H_{\ast}(H,\pi\otimes\chi)$ for any SAF representation $\pi$ of finite length of $G$ (Proposition \ref{prop:finite}).

Since $G$ is of inner type, there exists a canonical surjection \[(H\cap[G,G]_0)\backslash [G,G]_0/(P\cap [G,G]_0)\to H\backslash G/P.\]
By the description of the former double coset in \cite[Theorem 3]{Mat79}, we may choose a complete set of representatives $\{g_j\}_{j=1}^{j_0}$ for $H\backslash G/P$ so that $\sigma_j$ preserves $\fs$ for any $j\in[1,j_0]$.
Define an involution $\sigma_j$ of $G$ by
\[
\sigma_j:=\Psi(g_j)^{-1}\circ\sigma\circ\Psi(g_j).
\]

\begin{lemma}\label{lem:finite}
We have $d_{l,j,k}<\infty$ (see \eqref{df:d} for the definition) for any $j\in[1,j_0]$ and $l,k\in\N$.
\end{lemma}
\begin{proof}
Fix $j\in[1,j_0]$.
By \eqref{eq:pieces}, it suffices to show that $H_l(P^j,\gamma\otimes V)$ is finite-dimensional for any SAF representation $\gamma$ of finite length of $L$ and any finite-dimensional SF representation $V$ of $P^j$.

Let us prove it by induction on $l$.
By taking a composition series of $\gamma$ and applying Lemma \ref{lem:long exact}, we may assume that $\gamma$ is irreducible.

Write $P_L=L_LN_L$ for the standard minimal parabolic subgroup of $L$ and its Levi decomposition with $L_L\subset L$ and $N_L\subset N_{\min}$.
By Casselman's subrepresentation theorem, there exists a finite-dimensional SF representation $\gamma_L$ of $L_L$ such that 
the contragredient $\gamma^{\vee}$ is a subrepresentation of the induced representation $\Ind^L_{P_L}(\gamma_L^{\vee})$.
Hence $\gamma$ can be written as an irreducible quotient of $\Ind^L_{P_L}(\gamma_L)$.
Put $\gamma'$ to be the kernel of the quotient map, which has finite length as an SAF representation of $L$.
By Lemma \ref{lem:long exact}, we see 
\[
\dim H_l(P^j,\gamma\otimes V)\le \dim H_{l-1}(P^j,\gamma'\otimes V)+\dim H_l(P^j,\Ind^L_{P_L}(\gamma_L)\otimes V).
\]
From the induction hypothesis, it suffices to show that $H_l(P^j,\Ind^L_{P_L}(\gamma_L)\otimes V)$ is finite-dimensional for any finite-dimensional SF representation $\gamma_L$ of $L_L$ and any finite-dimensional SF representation $V$ of $P^j$.
We have isomorphisms
\begin{align}\notag
H_l(P^j,\Ind^L_{P_L}(\gamma_L)\otimes V)&\cong H_l(\fp^j_{\C},K\cap L^j;\Ind^L_{P_L}(\gamma_L)_{\text{$(K\cap L^j)$-fin}}\otimes V)\\\notag
&\cong\bigoplus_{0\le p\le l}H_p(\fl^j_{\C}\ltimes(\fl\cap\sigma_j\fn)_{\C},K\cap L^j;\Ind^L_{P_L}(\gamma_L)_{\text{$(K\cap L^j)$-fin}}\otimes H_{l-p}(\fn^j_{\C},V))\\
&\cong\bigoplus_{0\le p\le l}H_p(L^j\ltimes (L\cap \sigma_jN),\Ind^L_{P_L}(\gamma_L)\otimes H_{l-p}(\fn^j_{\C},V))\label{eq:reduction}
\end{align}
of vector spaces, where we used Lemma \ref{lem:comparison} at the first and the third isomorphisms, the Hochschild-Serre spectral sequences for the relative Lie algebra homology groups at the second isomorphism.

To prove the finite-dimensionality of \eqref{eq:reduction}, we replace $(G,H,P,\gamma,\chi^{\vee})$ in Section \ref{sec:5-1} by $(L,L^j\ltimes (L\cap \sigma_jN),P_L,\gamma_L,H_{l-p}(\fn^j_{\C},V))$ and apply the same arguments.
Recall that $L$ is a real reductive group of inner type, that $L^j$ is a symmetric subgroup of $L\cap\sigma_jL$, and that $(L\cap\sigma_jL)\ltimes(L\cap\sigma_jN)$ is a parabolic subgroup of $L$.
From the Bruhat decomposition and the double coset decomposition \cite[Theorem 3]{Mat79} for a symmetric group and a parabolic subgroup,
the cardinality of $(L^j\ltimes (L\cap \sigma_jN))\backslash L/P_L$ is finite and we may take a complete set $\{l_c\}_{1\le c\le c_0}$ of representatives so that the involution $\sigma_{j,c}:=\Psi(l_c)^{-1}\circ\sigma_j\circ\Psi(l_c)$ preserves $\fs$ for any $1\le c\le c_0$.
For a closed subgroup $R$ of $G$, we put $R^{j,c}:=R\cap l_c^{-1}g_j^{-1}Hg_jl_c$.
Then we see 
\[
\Psi(l_c)^{-1}(L^j\ltimes(L\cap\sigma_jN))=L^{j,c}\ltimes(L\cap\sigma_{j,c}N),\quad (L^{j,c}\ltimes(L\cap\sigma_{j,c}N))\cap P_L=L_L^{j,c}\ltimes (N_L^{j,c}(N_L\cap\sigma_{j,c}N)).
\]
From the finite-dimensionality of $\gamma_L$, Lemma \ref{lem:bound} and \eqref{eq:pieces},
it suffices to show that given $p\in\N$ and a finite-dimensional SF representation $W$ of $L_L^{j,c}\ltimes(N_L\cap\sigma_{j,c}N)$,
\begin{align}\label{eq:finite}
H_p(\fl^{j,c}_{L,\C}\ltimes(\fn_L^{j,c}+(\fn_L\cap\sigma_{j,c}\fn))_{\C},K\cap L^{j,c}_L;\Sym^k(\fl/(\fl^{j,c}+(\fl\cap\sigma_{j,c}\fn)+\fp_L))^{\vee}_{\C}\otimes W)=0
\end{align}
for sufficiently large $k\in\N$.

If $\fl=\fl^{j,c}+(\fl\cap\sigma_{j,c}\fn)+\fp_L$, then $\Sym^k(\fl/(\fl^{j,c}+(\fl\cap\sigma_{j,c}\fn)+\fp_L))^{\vee}_{\C}=0$ for $k\ge 1$ and there is nothing to prove.
Hence we may assume $\fl\neq\fl^{j,c}+(\fl\cap\sigma_{j,c}\fn)+\fp_L$.

As in the proof of Lemma \ref{lem:positive}, the $\fs^{\sigma_{j,c}}$-module $\fl$ is a submodule of $\fl^{j,c}+(\fl\cap\sigma_{j,c}\fn)+\fp_L+(\bar{\fn}_{\min}\cap\sigma_{j,c}\bar{\fn}_{\min})$.
By applying Lemma \ref{lem:positive} to the case $(G,\sigma,P)=(G,\sigma_{j,c},P_{\min})$, we obtain an element $X$ in $\fs^{\sigma_{j,c}}$ such that the eigenvalues of the action of $X$ on $\fn_L^{j,c}, \fn_L\cap\sigma_{j,c}\fn$ and $(\fl/(\fl^{j,c}+(\fl\cap\sigma_{j,c}\fn)+\fp_L))^{\vee}$ are all positive.
Hence the real parts of the eigenvalues of the action of $X$ on \[\wedge^{\ast}(\fn_L^{j,c}+(\fn_L\cap\sigma_{j,c}\fn))_{\C}\otimes\Sym^k(\fl/(\fl^{j,c}+(\fl\cap\sigma_{j,c}\fn)+\fp_L))^{\vee}_{\C}\otimes W\] are positive for sufficiently large $k$ from the assumption $\fl\neq\fl^{j,c}+(\fl\cap\sigma_{j,c}\fn)+\fp_L$.
Since $\fs^{\sigma_{j,c}}$ is contained in the center of $\fl^{j,c}_{L}$, we may apply Lemma \ref{lem:spectral seq} to the case \[(\fg,\fl,\fl',\fv,V)=(\fl_L,\fl_L,\fl_L^{j,c},\fn_L^{j,c}+(\fn_L\cap\sigma_{j,c}\fn),\C)\] and obtain \eqref{eq:finite}.
\end{proof}

\begin{lemma}\label{lem:asymptotic}
For $l\in\N$ and $j\in[1,j_0]$, we have $d_{l,j,k}=0$ for sufficiently large $k$.
\end{lemma}
\begin{proof}
Fix $l\in\N, j\in[1,j_0]$.
Since $H^j$ is reductive, the modular function $\delta_{H^j}$ of $H^j$ is trivial.
From \eqref{eq:pieces} and Lemma \ref{lem:factor}, $d_{l,j,k}$ equals the dimension of
\[
H_l((\fl^j+(\fl\cap\sigma_j\fn))_{\C}\ltimes\fn^j_{\C},K\cap L^j;\gamma_{\text{$(K\cap L^j)$-fin}}\otimes\Sym^k(\fg/(\fh^j+\fp))^{\vee}_{\C}\otimes{}^{g_j}\chi^{\vee}\cdot\delta_P^{1/2}\cdot\delta_{P^j}^{-1}).
\]

If $\fg=\fh^j+\fp$, then $\Sym^k(\fg/(\fh^j+\fp))^{\vee}_{\C}=0$ for $k\ge 1$ and there is nothing to prove.
Hence we may assume $\fg\neq\fh^j+\fp$.

By Lemma \ref{lem:positive}, there exists $X\in\fs^{\sigma_j}\cap\fz(\fl\cap\sigma_j\fl)$ where the eigenvalues of the adjoint actions of $X$ on $\fn^j, (\fg/(\fh^j+\fp))^{\vee}$ are all positive.
Lemma \ref{lem:n-homology2} shows that $H_p(\fl_{\C}\cap\sigma_j\fn_{\C},\gamma_{\text{$(K\cap L^j)$-fin}})$ is $Z(\fl_{\C}\cap\sigma_j\fl_{\C})$-finite.
Hence the real parts of the eigenvalues of the action of $X$ on
\[
H_p(\fl_{\C}\cap\sigma_j\fn_{\C},\gamma_{\text{$(K\cap L^j)$-fin}})\otimes\wedge^{\ast}\fn^j_{\C}\otimes
\Sym^k(\fg/(\fh^j+\fp))^{\vee}_{\C}\otimes{}^{g_j}\chi^{\vee}\cdot\delta_P^{1/2}\cdot\delta_{P^j}^{-1})
\]
are positive for sufficiently large $k$ from the assumption $\fg\neq\fh^j+\fp$.
Then the assertion follows from Lemma \ref{lem:spectral seq} (consider the case $(\fg,\fl,\fl',\fv,V)=(\fl,\fl\cap\sigma_j\fl,\fl^j,\fn^j,\gamma_{\text{$(K\cap L^j)$-fin}})$).
\end{proof}

\begin{proposition}\label{prop:finite}
Let $G$ be a real reductive group, $H$ a symmetric subgroup of $G$, $\pi$ an SAF representation $\pi$ of finite length of $G$, and $\chi$ a finite-dimensional representation of $H$.
For any $l\in\N$, $H_l(H,\pi\otimes\chi)$ is finite-dimensional.
\end{proposition}
\begin{proof}
The proof is similar to the first half of the proof of Lemma \ref{lem:finite}.
We may assume that $G$ and $H$ are connected (hence are of inner type), and that $\pi$ is irreducible. 
Let $P_{\min}=L_{\min}N_{\min}$ be the standard minimal parabolic subgroup of $G$.
The Casselman's subrepresentation theorem shows that $\pi$ can be written as a quotient of $\Pi=\Ind^G_{P_{\min}}(\gamma)$ for some irreducible representation $\gamma$ of $L_{\min}$.
Recall that any finite-dimensional continuous representation $\chi$ of the real reductive group $H$ is an SAF representation.
By Lemmas \ref{lem:bound}, \ref{lem:finite} and \ref{lem:asymptotic}, $H_l(H,\Pi\otimes\chi)$ is finite-dimensional for any $l\in\N$.
Then the long exact sequence (Lemma \ref{lem:long exact}) and the induction argument show the finite-dimensionality of $H_l(H,\pi\otimes\chi)$ for any $l\in\N$.
\end{proof}

\subsection{Normal derivatives and stable Levi subgroups}\label{sec:5-3}
Retain notation in Section \ref{sec:5-2}.
In this section,  we prove vanishing of $d_{l,j,k}$ assuming the following conditions on $(G,H,P,\gamma,\chi)$.
\begin{enumerate}[(A)]
\item\label{cond:rho}
Any simple component of $(\fg,\fh)$ is one of the following symmetric pairs;
\begin{enumerate}[(1)]
\item $(\fg'\oplus\fg',\fg')$ or $(\fg'_{\C},\fg')$ where $\fg'$ denotes a real simple Lie algebra,
\item $(\mathfrak{sl}_{2n}(\R),\mathfrak{sl}_n(\C)+\sqrt{-1}\R)$ or $(\mathfrak{sl}_n(\H),\mathfrak{sl}_n(\C)+\sqrt{-1}\R)$, or
\item $(\fe_{6(6)},\mathfrak{sl}_3(\H)\oplus\mathfrak{su}(2))$ or $(\fe_{6(-26)},\mathfrak{sl}_3(\H)\oplus\mathfrak{su}(2))$.
\end{enumerate}
\item\label{cond:positive}
The parabolic subgroup $P$ is cuspidal, and the representation $\gamma$ of $L=MA$ is written as $\tau_{\lambda}=\tau\boxtimes\exp(\lambda)$, where $\tau$ denotes a square integrable SAF representation of $M$ and $\lambda$ denotes an element of $\fa^{\vee}_{\C}$ with $\re\lambda(X_{\alpha})\ge 0$ for any $\alpha\in\Sigma^+$ (see the sentence before Lemma \ref{lem:positive} for the definition of $X_{\alpha}$).
\item\label{cond:gamma}
Any simple ideal of $\fl$ is compact or $\mathfrak{sl}_2(\R)$.
\item\label{cond:char}
The finite-dimensional representation $\chi$ of $H$ is a character.
\end{enumerate}

\begin{remark}
For any linear connected semisimple Lie group $G$, any irreducible SAF representation of $G$ can be written as a quotient of $\Ind^G_P(\gamma)$ with $(P,\gamma)$ satisfying the condition \eqref{cond:positive} above (see \cite[Theorem 14.92]{Kna01} for example).
\end{remark}

\begin{remark}\label{rem:cuspidal}
Assume $\fg=\fg_{\C}',\sl_n(\H),\fe_{6(-26)}$, where $\fg'$ denotes a real simple Lie algebra.
Since cuspidal parabolic subalgebras of $\fg$ are minimal, \eqref{cond:positive} implies \eqref{cond:gamma}.
\end{remark}

Let us define
\begin{align*}
\cJ&:=\Set{j\in[1,j_0]|\sigma_j\fl=\fl}.
\end{align*}
If $j\in\cJ$, then $\sigma_j$ preserves $L$ and the character $\delta_{L\cap\sigma_jN}$ of $P^j$ is trivial (see the sentence before Lemma \ref{lem:rho} for the definition).
\begin{theorem}\label{thm:bound}
Assume conditions \eqref{cond:rho}--\eqref{cond:char}.
Then we have 
\[\dim\Hom_H(\Ind^G_P(\gamma),\chi)\le \sum_{j\in\cJ}\dim\Hom_{L^{\sigma_j}}(\gamma,{}^{g_j}\chi).\]
\end{theorem}

\begin{proof}
By Lemma \ref{lem:Hausdorff} and Proposition \ref{prop:finite}, we have 
\[
H_0(H,\Ind^G_P(\gamma)\cdot\chi^{-1})^{\vee}\cong\Hom_H(\Ind^G_P(\gamma),\chi),\quad H_0(L^{\sigma_j},\gamma\cdot{}^{g_j}\chi^{-1})^{\vee}\cong\Hom_{L^{\sigma_j}}(\gamma,{}^{g_j}\chi).\]
From Lemma \ref{lem:bound} and Proposition \ref{prop:finite},
it suffices to prove 
$d_{0,j,k}=0$ if $j\in[1,j_0]\setminus\cJ$ or $k>0$.
As in the proof of Lemma \ref{lem:asymptotic}, $d_{0,j,k}$ equals the dimension of 
\[
H_0((\fl^j+(\fl\cap\sigma_j\fn))_{\C}\ltimes\fn^j_{\C},K\cap L^j;\gamma_{\text{$(K\cap L^j)$-fin}}\otimes\Sym^k(\fg/(\fh^j+\fp))^{\vee}_{\C}\otimes{}^{g_j}\chi^{-1}\cdot\delta_P^{1/2}\cdot\delta_{P^j}^{-1}).
\]

By multiplying $g_j$ $(1\le j_0)$ by elements of $L$ from the right if necessary, we may assume
\begin{align}\label{eq:lpositive}
\Sigma_L\cap\sigma_j\Sigma_N\subset\Sigma^+.
\end{align}
We remark that this does not change the definition of $\cJ, d_{l,j,k}$ for any $l,j,k$.
Then there exists $X\in\fs^{\sigma_j}\cap[\fg,\fg]\cap\fz(\fl\cap\sigma_j\fl)$ such that
\begin{enumerate}[(i)]
\item\label{ind:iii}
$X$ belongs to $\sum_{\alpha\in\Sigma^+}\R_{>0}X_{\alpha}$,
\item\label{ind:i}
the eigenvalues of the adjoint action of $X$ on $(\fg/(\fh^j+\fp))^{\vee}_{\C}$ are all positive,
\item\label{ind:ii}
$\alpha(X)>0$ for any $\alpha\in\Sigma_L\cap\sigma_j\Sigma_N$
\end{enumerate}
from Lemma \ref{lem:positive}.
Our assumption \eqref{cond:rho} shows $\delta_P^{1/2}\cdot\delta_{P^j}^{-1}=\delta_{L\cap \sigma_jN}^{-1/2}$ and $\fs^{\sigma_j}\cap[\fg,\fg]=\fs^{\sigma_j}\cap[\fh^j,\fh^j]$ by Lemma \ref{lem:rho}.
By the assumption \eqref{cond:char}, $X\in\fs^{\sigma_j}\cap[\fg,\fg]=\fs^{\sigma_j}\cap[\fh^j,\fh^j]$ acts trivially on ${}^{g_j}\chi$.
Therefore Lemma \ref{lem:spectral seq} reduces the assertion to proving that the eigenvalues of the action of $X$ on 
\begin{align}\label{eq:target}
H_0(\fl_{\C}\cap\sigma_j\fn_{\C},\gamma_{\text{$(K\cap L^j)$-fin}})\otimes\Sym^k(\fg/(\fh^j+\fp))^{\vee}_{\C}\otimes\delta_{L\cap \sigma_jN}^{-1/2}
\end{align}
are nonzero if $j\in[1,j_0]\setminus\cJ$ or $k>0$.

Since $K\cap L^j$ has finitely many connected components, the assumption \eqref{cond:positive} shows
\begin{align}\label{eq:target2}
H_0(\fl_{\C}\cap\sigma_j\fn_{\C},\gamma_{\text{$(K\cap L^j)$-fin}})=H_0(\fm_{\C}\cap\sigma_j\fn_{\C},\tau_{\text{$(\fk\cap\fl^j)$-fin}})\otimes\exp(\lambda)
\end{align}
as $(\fl^j_{\C},(K\cap L^j)\times A)$-modules, where $(\cdot)_{\text{$(\fk\cap\fl^j)$-fin}}$ denotes the subspace of $U(\fk_{\C}\cap\fl^j_{\C})$-finite vectors.
Let us consider the eigenvalues of the action of $X$ on $H_0(\fm_{\C}\cap\sigma_j\fn_{\C},\tau_{\text{$(\fk\cap\fl^j)$-fin}})$.
Let $\fii$ be a noncompact simple ideal in $\fl$. By the assumption \eqref{cond:gamma}, $\fii\cong\sl_2(\R)$ and $\fii^{\theta}\cong\so(2)$.
By the assumption \eqref{cond:rho} and Remark \ref{rem:cuspidal}, the ideal $\fii$ is contained in a simple component of the symmetric pair $(\fg,\fh)$ isomorphic to $(\fg'\oplus\fg',\fg'), (\sl_{2n}(\R),\sl_n(\C)+\sqrt{-1}\R)$ or $(\fe_{6(6)},\sl_3(\H)\oplus\mathfrak{su}(2))$.
Since $\fii$ has real rank one, either of the following holds.
\begin{enumerate}[(I)]
\item\label{ind:1}
$\sigma_j\fii\cap\fn\neq 0$.
\item\label{ind:3}
$\sigma_j\fii$ is an ideal of $\fl$ and $\sigma_j\fii\neq\fii$.
\item\label{ind:2}
$\sigma_j\fii=\fii$. In this case, $\sigma_j$ acts on $\fii\cap\fs$ by the scalar multiple $(-1)$ by Lemma \ref{lem:extension} and the space $\fii^{\sigma_j}\cap\fk$ is one-dimensional.
\end{enumerate}
Set $n_1, 2n_2, n_3$ to be the number of noncompact simple ideals $\fii$ satisfying the condition \eqref{ind:1}, \eqref{ind:3}, \eqref{ind:2}, respectively.
We remark that $n_1=0$ if and only if $j\in\cJ$.
Then we obtain the following decomposition
\begin{align}
\fm&=\fm_{\cpt}\oplus\fm^1\oplus\fm^2\oplus\fm^3,\\
\fm^1=\bigoplus_{1\le i\le n_1}\sl_2(\R)^1_i,\quad
\fm^2&=\bigoplus_{1\le i\le n_2}\sl_2(\R)_{i,1}^2\oplus\sl_2(\R)^2_{i,2},\quad
\fm^3=\bigoplus_{1\le i\le n_3}\sl_2(\R)^3_i,
\end{align}
where $\fm_{\cpt}, \{\sl_2(\R)^1_i\}_{1\le i\le n_1}, \{\sl_2(\R)^3_i\}_{1\le i\le n_3}$ denote the maximal compact ideal in $\fm$, the noncompact simple ideals of $\fl$ satisfying \eqref{ind:1}, and those satisfying \eqref{ind:2}, respectively, and 
$\sigma_j$ interchanges $\sl_2(\R)^2_{i,1}$ and $\sl_2(\R)^2_{i,2}$ for $1\le i\le n_2$.
Put $\fn_i$ to be the nilradical of the minimal parabolic subalgebra of $\sl_2(\R)^1_i$ corresponding to $\Sigma(\sl_2(\R)^1_i,\sl_2(\R)^1_i\cap\fs)\cap\Sigma^+$ for $1\le i\le n_1$.
By the assumption \eqref{eq:lpositive}, $\fn_i=\sl_2(\R)^1_i\cap\sigma_j\fn$. 
We see
\begin{align}\label{eq:m}
\fm\cap\sigma_j\fn=\bigoplus_{1\le i\le n_1}\fn_i,\quad
\fk\cap\fl^j=(\fm_{\cpt})^j\oplus
(\fk\cap\fm^2)^j
\oplus\left(\bigoplus_{1\le i\le n_3}\so(2)^3_i\right).
\end{align}

Since $M$ has finitely many components, the irreducible SAF representation $\tau$ of $M$ can be written as a finite direct sum of the form
\[
\tau_0=\tau_{\cpt}\otimes\tau^1\hat{\otimes}\tau^2\hat{\otimes}\tau^3
\]
as representations of $M_0$, where $\hat{\otimes}$ denotes the tensor product of nuclear Fr\'echet spaces, $\tau_{\cpt}$ denotes an irreducible finite-dimensional $\fm_{\cpt}$-module and 
\[
\tau^1=\widehat{\bigotimes_{1\le i\le n_1}}\tau^1_i,\quad\quad
\tau^2=\widehat{\bigotimes_{1\le i\le n_2}}\tau^2_{i,1}\hat{\otimes}\tau^2_{i,2},\quad\quad
\tau^3=\widehat{\bigotimes_{1\le i\le n_3}}\tau^3_i
\]
for some square integrable SAF representations $\tau^1_i,\tau^2_{i,1},\tau^2_{i,2},\tau^3_i$ of finite covers of $\PSL_2(\R)$.
From \eqref{eq:m}, the representation $\tau^3$ is $(K\cap L^j)_0$-admissible. Hence 
\begin{align*}
(\tau_0)_{\text{$(\fk\cap\fl^j)$-fin}}
=\tau_{\cpt}\otimes(\tau^1\hat{\otimes}
\tau^2)_{\text{$(\fk\cap\fl^j)$-fin}}\otimes
\tau^3_{\text{$(\fk\cap\fl^j)$-fin}}.
\end{align*}
Write $\cU$ for the unitary dual of the analytic subgroup of $\fk\cap\fm^2$. 
For $\xi\in\cU$, write $\tau^2(\xi)$ for the $\xi$-isotypic component of $\tau^2$, which is a closed subspace.
Then we see 
$
(\tau^1\hat{\otimes}\tau^2)_{\text{$(\fk\cap\fl^j)$-fin}}=\bigoplus_{\xi\in\cU}\tau^1\hat{\otimes}\tau^2(\xi)
$ 
and obtain
\begin{align}\label{eq:tensor}
H_0(\fm_{\C}\cap\sigma_j\fn_{\C},(\tau_0)_{\text{$(\fk\cap\fl^j)$-fin}})
=\bigoplus_{\xi\in\cU}\tau_{\cpt}\otimes H_0(\fm_{\C}\cap\sigma_j\fn_{\C},\tau^1\hat{\otimes}\tau^2(\xi))\otimes\tau^3_{\text{$(\fk\cap\fl^j)$-fin}}.
\end{align}

Since $\fm\cap\sigma_j\fn$ is the nilradical of the minimal parabolic subalgebra of $\fm^1$ corresponding to $\Sigma^+$, Casselman's unpublished result (for the case of minimal parabolic groups, see \cite[Theorem 5.2]{LLY21}) implies
\begin{align*}
H_0(\fm_{\C}\cap\sigma_j\fn_{\C},\tau^1)
=H_0(\fm_{\C}\cap\sigma_j\fn_{\C},\tau^1_{\text{$(\fk\cap\fm^1)$-fin}})
=\bigotimes_{1\le i\le n_1}H_0(\fn_{i,\C},(\tau^1_i)_{\text{$\so(2)^1_i$-fin}}),
\end{align*}
where we used \eqref{eq:m} and 
$\tau^1_{\text{$(\fk\cap\fm^1)$-fin}}=\bigotimes_{1\le i\le n_1}(\tau^1_i)_{\text{$\so(2)_i$-fin}}$
at the last equality.
Since $\tau^1_i$ is a square integrable representation of the analytic subgroup of $\sl_2(\R)^1_i$, which is a finite cover of $\PSL_2(\R)$, there exist $\epsilon_i\in\{\pm\}$ and $k_i\in\Q_{>0}$ such that the infinitesimal character of $\tau^1_i$ equals $k_i\rho_{\fn_i}$ and 
\[(\tau^1_i)_{\text{$\so(2)^1_i$-fin}}=\bigoplus_{l\in\N}\xi_{\epsilon_i(k_i+1+2l)}\]
as $\so(2)^1_i$-modules (see \cite[Section XI.8]{KV95}, for example). 
See Section \ref{sec:intro-3} for the definition of $\xi_{\epsilon(1+k_i+2l)}$.
Then we see that $H_0(\fn,(\tau^1_i)_{\text{$\so(2)^1_i$-fin}})$ is one-dimensional and its $\fs$-weight equals $(k_i+1)\rho_{\fn_i}$ from a general result concerning zeroth $\fn$-homology groups of irreducible highest weight modules \cite[Lemma 2.7]{Wal84}.

Recall that tensoring with the nuclear Fr\'echet space $\tau^2(\xi)$ preserves short exact sequences of nuclear Fr\'echet spaces with continuous maps \cite[Lemma A.3]{CHM00}, and that the Schwartz homology groups can be written as the homology groups of complexes of nuclear Fr\'echet spaces \eqref{eq:standard}.
Therefore we have
\begin{align}\label{eq:example}
H_0(\fm_{\C}\cap\sigma_j\fn_{\C},\tau^1\hat{\otimes}\tau^2(\xi))
&\cong H_0(M\cap\sigma_jN,\tau^1\hat{\otimes}\tau^2(\xi))\\\notag
&\cong H_0(M\cap\sigma_jN,\tau^1)\otimes\tau^2(\xi)
\cong H_0(\fm_{\C}\cap\sigma_j\fn_{\C},\tau^1)\otimes\tau^2(\xi)
\end{align}
from the finite-dimensionality of $H_0(\fm_{\C}\cap\sigma_j\fn_{\C},\tau^1)$ and Lemmas \ref{lem:comparison} and \ref{lem:Hausdorff}. Hence \eqref{eq:tensor} shows
\begin{align*}
H_0(\fm_{\C}\cap\sigma_j\fn_{\C},(\tau_0)_{\text{$(\fk\cap\fl^j)$-fin}})
=\tau_{\cpt}\otimes H_0(\fm_{\C}\cap\sigma_j\fn_{\C},\tau^1)\otimes\tau^2_{\text{$(\fk\cap\fl^j)$-fin}}\otimes\tau^3_{\text{$(\fk\cap\fl^j)$-fin}}.
\end{align*}
Therefore the element $X$ in $\fs^{\sigma_j}\cap\fz(\fl\cap\sigma_j\fl)$ acts on 
$H_0(\fm_{\C}\cap\sigma_j\fn_{\C},(\tau_0)_{\text{$(\fk\cap\fl^j)$-fin}})\otimes\delta_{L\cap\sigma_jN}^{-1/2}$ by a scalar $\sum_{1\le i\le n_1}(k_i+1)\rho_{\fn_i}(X)-\rho_{\fl\cap\sigma_j\fn}(X)=\sum_{1\le i\le n_1}k_i\rho_{\fn_i}(X)$.

From the above arguments and \eqref{eq:target2}, the eigenvalues of $X$ on \eqref{eq:target} are of the form 
\begin{align}\label{eq:scalar}
\sum_{1\le i\le n_1}k_i\rho_{\fn_i}(X)+\lambda(X)+\sum_{1\le i\le k}\alpha_i(X),
\end{align}
where $\alpha_i$ denotes an $\fs^{\sigma_j}$-weight appearing in $(\fg/(\fh^j+\fp))^{\vee}$.
Now $k_i>0$ and \eqref{ind:ii} implies $k_i\rho_{\fn_i}(X)>0$ for any $i\in[1,n_1]$.
Moreover, $\re \lambda(X)\ge 0$ by the assumption \eqref{cond:positive} and \eqref{ind:iii}.
Furthermore, \eqref{ind:i} shows $\alpha_i(X)>0$ for any $i\in[1,k]$.
Therefore \eqref{eq:scalar} is nonzero unless $n_1=k=0$, which is the desired conclusion.
\end{proof}

\subsection{Proof of Theorem \ref{thm:key-intro} and Corollary \ref{cor:dual}}\label{sec:5-4}
Let $(G,H)=(\GL_{2n}(\R),\GL_n(\C))$ or $(\GL_n(\H),\GL_n(\C))$.
In this section, we first recall from \cite{Cho19}  
the orbit decomposition $H\backslash G/P$ and the description of the associated involutions $\{\sigma_j\}$ on $G$.
Then we apply Theorem \ref{thm:bound} to obtain Theorem \ref{thm:key-intro}. Moreover, we prove Corollary \ref{cor:dual}.

We take a Cartan involution $\theta$ on $G$ given by $\theta(g)={}^t\bar{g}^{-1}$ for $g\in G$ and set $K=G^{\theta}$.
Moreover, we define an involution $\sigma$ on $G$ by $\Psi\left(\begin{pmatrix}
0&-J_n\\
J_n&0
\end{pmatrix}\right)$ if $G=\GL_{2n}(\R)$, and by $\Psi(\bfi I_n)$ if $G=\GL_n(\H)$, where $J_n$ denotes the anti-diagonal $n$-by-$n$ matrix with anti-diagonal entries one and $I_n$ denotes the $n$-by-$n$ unit matrix.
The group $H$ is identified with $G^{\sigma}$ via the embeddings $\C^n\to\R^{2n};(z_1,\ldots,z_n)\mapsto (\re(z_1),\ldots,\re(z_n),\im(z_n),\ldots,\im(z_1))$ and $\C\to\H;a+b\sqrt{-1}\mapsto a+b\bfi$.

Let $(n_1,\ldots,n_r)$ be a partition of $2n$ with $n_i\in\{1,2\}$ if $G=\GL_{2n}(\R)$, and a partition of $n$ with $n_i=1$ if $G=\GL_n(\H)$. In particular, we see $r=n$ if $G=\GL_n(\H)$.
We define
\begin{align*}
\fI:=
\begin{cases}
\Set{S\in M_r(\Z_{\ge 0})|{}^tS=S, \sum_{1\le j\le r}S_{i,j}=n_i, S_{i,i}\in 2\Z\text{ for $1\le i\le r$}}&\text{ if $G=\GL_{2n}(\R)$,}\\
\Set{\text{permutation matrices $S\in M_n(\Z_{\ge 0})$ with $S^2=1$}}&\text{ if $G=\GL_{n}(\H)$.}
\end{cases}
\end{align*}
Given $S\in\fI$, we set $g_S\in K$ as follows.

When $G=\GL_{2n}(\R)$, let $I_{i,j}$ ($1\le i,j\le r$) be the $S_{i,j}$-tuples defined by 
\[
(1,2,\ldots,2n)=\prod_{i=1}^r\prod_{j=1}^rI_{i,j}=I_{1,1}\cdot I_{1,2}\cdot I_{1,3}\cdots I_{r,r},
\]
and $I^+_{i,j}, I^-_{i,j}$ the tuples defined by 
\[
I_{i,j}=I^+_{i,j}\cdot I^-_{i,j},\quad I^-_{i,j}=()\text{ if $i<j$},\quad I^+_{i,j}=()\text{ if $i>j$},\text{ and $I^+_{i,i}, I^-_{i,i}$ are $S_{i,i}/2$-tuples,}
\]
where $()$ denotes the $0$-tuple.
For example, $I_{1,1}=(1,2,\ldots,S_{1,1})$, $I_{1,2}=(S_{1,1}+1,S_{1,1}+2,\ldots,S_{1,1}+S_{1,2})$ and $I^+_{1,1}=(1,2,\ldots,S_{1,1}/2)$.
In this case, we set $g_S\in K$ to be the permutation matrix satisfying
\[
(g_S^{-1}(1),\ldots,g_S^{-1}(2n))=\left(\prod_{i=1}^r\prod_{j=i}^rI^+_{i,j}\right)\cdot\left(\prod_{j=1}^r\prod_{i=1}^jI^-_{r+1-i,r+1-j}\right).
\]

When $G=\GL_n(\H)$, let $V_{i,j}$ be the right $\H$-vector space spanned by 
\[
\Set{e_a\in\H^n|\sum_{k=1}^{i-1}\sum_{l=1}^rS_{k,l}+\sum_{l=1}^{j-1}S_{i,l}< a\le \sum_{k=1}^{i-1}\sum_{l=1}^rS_{k,l}+\sum_{l=1}^{j}S_{i,l}},
\]
where $\{e_1,\ldots,e_n\}$ denotes the canonical basis of the right $\H$-vector space $\H^n$. In this case, we set $g_S\in K$ by
\[
g_S=1\text{ on $V_{i,i}$ ($1\le i\le r$)},\quad g_S=\begin{pmatrix}1&1\\\bfj&-\bfj\end{pmatrix}\text{ on $V_{i,j}+V_{j,i}$ ($i\neq j$)}.
\]

Let $P=LN$ be the standard parabolic subgroup of $G$ consisting of upper triangular block matrices corresponding to the partition $(n_1,\ldots,n_r)$ and its Levi decomposition.
\begin{lemma}[{\cite[Propositions 2.1 and 3.1]{Cho19}}]\label{lem:orbit}
The map
$\fI\to H\backslash G/P;S\mapsto Hg_SP$ is bijective.
\end{lemma}

Set $\sigma_S:=\Psi(g_S)^{-1}\circ\sigma\circ\Psi(g_S)$ for $S\in\fI$.
Now the involution $\sigma_S$ on $G$ preserves $L$ if and only if $S$ is a monomial matrix, that is, an invertible matrix of the form $S_1S_2$ where $S_1$ (resp. $S_2$) denotes a diagonal matrix (resp. a permutation matrix).
By considering the permutation associated to $S$, 
we may identify the set of such matrices with
\begin{align*}
\fS:=
\begin{cases}
\Set{\varsigma\in\fS_r|\varsigma^2=1,\quad n_i=n_{\varsigma(i)}\text{ for $1\le i\le r$,}\quad n_i=2\text{ if $\varsigma(i)=i$}}
&\text{ if $G=\GL_{2n}(\R)$,}\\
\Set{\varsigma\in\fS_n|\varsigma^2=1}&\text{ if $G=\GL_n(\H)$.}
\end{cases}
\end{align*}
For $\varsigma\in\fS$, let us write $\sigma_{\varsigma}$ (resp.\,$g_{\varsigma}$) for the involution $\sigma_{S}$ preserving $L$ (resp.\,$g_S$), where $S$ is the element in $\fI$ corresponding to $\varsigma$.
\begin{lemma}[{\cite[pages 206 and 214]{Cho19}}]\label{lem:orbit2}
Let $\varsigma\in\fS$.
Under the block-wise decomposition $L=\prod_{i=1}^rL_i$ ($L_i=\GL_{n_i}(\R)$ if $G=\GL_{2n}(\R)$ and $L_i=\GL_{n_i}(\H)$ if $G=\GL_n(\H)$), the isomorphism $\upsilon_i=\sigma_{\varsigma}|_{L_i}$ from $L_i$ onto $L_{\varsigma(i)}$ is written as
\[\upsilon_i=
\begin{cases}
\id&\text{ if $G=\GL_{2n}(\R)$ and $n_i=1$},\\
\Psi\left(\begin{pmatrix}0&-1\\1&0\end{pmatrix}\right)&\text{ if $G=\GL_{2n}(\R), n_i=2$ and $\varsigma(i)=i$},\\
\Psi(J_2)&\text{ if $G=\GL_{2n}(\R), n_i=2$ and $\varsigma(i)\neq i$},\\
\Psi(\bfi)&\text{ if $G=\GL_n(\H)$}.
\end{cases}
\]
\end{lemma}

For $1\le i\le r$ with $i\le\varsigma(i)$, define an injective homomorphism 
\begin{align*}
\Upsilon_i\colon 
\begin{cases}
L_i^{\sigma_{\varsigma}}\to L^{\sigma_{\varsigma}};x\mapsto
\diag(I_{n_1},\ldots,I_{n_{i-1}},x,I_{n_{i+1}},\ldots,I_{n_r})&\text{ if $i=\varsigma(i)$,}\\
L_i\to L^{\sigma_{\varsigma}};x\mapsto
\diag(I_{n_1},\ldots,I_{n_{i-1}},x,1,I_{n_{i+1}},\ldots,I_{n_{\varsigma(i)-1}},\upsilon(x),I_{n_{\varsigma(i)+1}},\ldots,I_{n_r})&\text{ if $i<\varsigma(i)$.}
\end{cases}
\end{align*}
A simple computation gives the following 
\begin{lemma}\label{lem:chi}
The composite of $\det_{\GL_n(\C)}\circ \Psi(g_{\varsigma})\colon L^{\sigma_{\varsigma}}\to L^{\sigma}\to\C^{\times}$ and $\Upsilon_i$ equals $\det_{L_i^{\sigma_{\varsigma}}}$ if $i=\varsigma(i)$, $\det_{L_i}$ if $i<\varsigma(i)$.
\end{lemma}

We set
\[
\fT:=
\begin{cases}
\Set{\varsigma\in\fS_{r}|
\begin{array}{ll}
\varsigma^2=1,n_i=n_{\varsigma(i)}\text{ for $1\le i\le r$, and}\\
\text{$n_i=2$ and $\pi_i$ is $\chi_{\GL_1(\C)}$-distinguished}&\text{ if $\varsigma(i)=i$,}\\
\text{$\pi_i\hat{\otimes}\pi_{\varsigma(i)}$ is $\chi_{\GL_{n_i}(\R)}$-distinguished}&\text{ if $\varsigma(i)\neq i$}
\end{array}
}
&\text{ if $G=\GL_{2n}(\R)$},
\\
\Set{\varsigma\in\fS_n|
\begin{array}{ll}
\varsigma^2=1,\\
\text{$\pi_i$ is $\chi_{\GL_1(\C)}$-distinguished}&\text{ if $\varsigma(i)=i$,}\\
\pi_i\otimes\pi_{\varsigma(i)}\text{ is $\chi_{\GL_1(\H)}$-distinguished}&\text{ if $\varsigma(i)\neq i$}
\end{array}
}
&\text{ if $G=\GL_n(\H)$}.
\end{cases}
\]

\begin{remark}\label{rem:parameter}
Let us write the character $\chi$ of $\C^{\times}$ as 
$\chi(z)=(z/|z|)^{l}|z|^{\eta}$ with $l\in\Z$, $\eta\in\C$.
For any irreducible SAF representations $\pi_1, \pi_2$ of a general linear group $G'$ over $\R$ or $\H$, the representation $\pi_1\hat{\otimes}\pi_2$ is $\chi_{G'}$-distinguished if and only if $\pi_1\cong\pi_2^{\vee}\cdot\chi_{G'}$ by Lemma \ref{lem:contragredient}, 
which is also equivalent to
\[
\lambda_1+\lambda_2=\eta,
\begin{cases}
k_1+k_2\in l+2\Z&\text{ when $(G',\pi_i)=(\GL_1(\R),\pi^{(1)}_{k_i,\lambda_i})$,}\\
k_1=k_2&\text{ when $(G',\pi_i)=(\GL_2(\R),\pi^{(2)}_{k_i,\lambda_i}),(\GL_1(\H),\tau_{k_i,\lambda_i})$}
\end{cases}
\]
where $(k_i, \lambda_i)$ belongs to $\{0,1\}\times\C$ in the former case, and to $\Z_{\ge 1}\times\C$ in the latter case for $i=1,2$.
Moreover, the representation $\pi=\pi^{(2)}_{k,\lambda}$ (resp.\,$\tau_{k,\lambda}$) of $G'=\GL_2(\R)$ (resp.\,$\GL_1(\H)$) with $k\in\Z_{\ge 1},\lambda\in\C$ is $\chi_{\GL_1(\C)}$-distinguished if and only if
\[
2\lambda=\eta,\quad
\begin{cases}
k+1\in |l|-2\Z_{\ge 0}&\text{ when $G'=\GL_2(\R)$,}\\
k-1\in |l|+2\Z_{\ge 0}&\text{ when $G'=\GL_1(\H)$}
\end{cases}
\]
as we saw in Section \ref{sec:2-2}. In this case, $\pi\cong\pi^{\vee}\cdot\chi_{G'}$ also holds.
\end{remark}

Theorem \ref{thm:key-intro} follows from the following
\begin{theorem}\label{thm:key}
We have $\dim \Hom_H(\pi_1\times\cdots\times\pi_r,\chi_H)\le \#\fT.$
In particular, if $\pi_1\times\cdots\times\pi_r$ is $\chi_H$-distinguished, then $\fT$ is not empty.
\end{theorem}

\begin{proof}
Recall that $(G,H,P,\pi_1\hat{\otimes}\cdots\hat{\otimes}\pi_r,\chi_H)$ satisfies the conditions \eqref{cond:rho}--\eqref{cond:char}.
Therefore, by Theorem \ref{thm:bound} and Lemma \ref{lem:orbit}, we have
\begin{align}\label{eq:key}
&\dim\Hom_H(\pi_1\times\cdots\times\pi_r,\chi_H)\le\sum_{\varsigma\in\fS}\dim H_0(\fl^{\sigma_{\varsigma}}_{\C},K\cap L^{\sigma_{\varsigma}};(\pi_1\hat{\otimes}\cdots\hat{\otimes}\pi_r)_{\text{$(K\cap L^{\sigma_{\varsigma}})$-fin}}\cdot{}^{g_{\varsigma}}\chi_H^{-1}).
\end{align}
Write $I_{\varsigma}=\Set{1\le i\le r|i<\varsigma(i)}$.
Since $\pi_i$ is $(K\cap L_i)$-admissible for any $i\in[1,r]^{\varsigma}$, we have
\[
(\pi_1\hat{\otimes}\cdots\hat{\otimes}\pi_r)_{\text{$(K\cap L^{\sigma_{\varsigma}})$-fin}}\cong\bigotimes_{i\in[1,r]^{\varsigma}}(\pi_i)_{\text{$(K\cap L_i)$-fin}}\otimes\left(\widehat{\bigotimes_{i\in I_{\varsigma}}}\pi_i\hat{\otimes}\pi_{\varsigma(i)}\right)_{\text{$(K\cap L_i)$-fin}}.
\]
From Lemma \ref{lem:orbit2}, we have $\pi_{\varsigma(i)}\circ\upsilon_i\cong\pi_{\varsigma(i)}$ for any $i\in I_{\varsigma}$.
Therefore Lemma \ref{lem:chi} shows that the right hand side of \eqref{eq:key} equals 
\[
\sum_{\varsigma\in\fS}\dim\left(\bigotimes_{i\in[1,r]^{\varsigma}} H_0((\fl_i^{\sigma_{\varsigma}})_{\C},K\cap L_i^{\sigma_{\varsigma}};(\pi_i)_{\text{$(K\cap L_i)$-fin}}\cdot\chi^{-1}_{L_i^{\sigma_{\varsigma}}})\right)\cdot\dim\cA_{\varsigma},
\]
where 
\begin{align*}
\cA_{\varsigma}&:=
H_0\left(\bigoplus_{i\in I_{\varsigma}}(\fl_i)_{\C},K\cap \prod_{i\in I_{\varsigma}}L_i;\left(\widehat{\bigotimes_{i\in I_{\varsigma}}}(\pi_i\hat{\otimes}\pi_{\varsigma(i)})\cdot\chi^{-1}_{L_i}\right)_{\text{$(K\cap L_i)$-fin}}\right)\\
&\cong
H_0\left(\prod_{i\in I_{\varsigma}}L_i,\widehat{\bigotimes_{i\in I_{\varsigma}}}(\pi_i\hat{\otimes}\pi_{\varsigma(i)})\cdot\chi^{-1}_{L_i}\right)\\
&\cong
\bigotimes_{i\in I_{\varsigma}}H_0(L_i,(\pi_i\hat{\otimes}\pi_{\varsigma(i)})\cdot\chi^{-1}_{L_i}).
\end{align*}
Here we used Lemma \ref{lem:comparison} for the first isomorphism,  and Proposition \ref{prop:finite} and similar arguments as \eqref{eq:example} for the second one.
Hence again by Lemmas \ref{lem:comparison} and \ref{lem:Hausdorff} and Proposition \ref{prop:finite}, the right hand side of \eqref{eq:key} equals
\begin{align*}
\sum_{\varsigma\in\fS}\left(\prod_{i\in[1,r]^{\varsigma}}\dim\Hom_{L_i^{\sigma_{\varsigma}}}(\pi_i,\chi_{L_i^{\sigma_{\varsigma}}})\right)\left(\prod_{i\in I_{\varsigma}}\dim\Hom_{L_i}(\pi_i\hat{\otimes}\pi_{\varsigma(i)},\chi_{L_i})\right).
\end{align*}

When $i\in[1,r]^{\varsigma}$, we have $\dim\Hom_{L_i^{\sigma_{\varsigma}}}(\pi_i,\chi_{L_i})\le 1$ since $(\pi_i)_{\text{$(K\cap L_i^{\sigma_{\varsigma}})$-fin}}$ is $(K\cap L_i^{\sigma_{\varsigma}})$-multiplicity free.
When $i\in I_{\varsigma}$, we have $\dim\Hom_{L_i}(\pi_i\hat{\otimes}\pi_{\varsigma(i)},\chi_{L_i})\le 1$ by Lemma \ref{lem:contragredient}.
From the above arguments, the right hand side of \eqref{eq:key} equals $\#\fT$, and the proof is complete.
\end{proof}

\begin{proof}[Proof of Corollary \ref{cor:dual}]
Assume that an irreducible SAF representation $\pi$ of $G$ is $\chi_H$-distinguished. 
Let $\widetilde{\pi}=\pi_1\times\pi_2\times\cdots\times\pi_r$ be a standard module above $\pi$, $(n_1,\ldots,n_r)$ the corresponding partition, $\phi_\pi$ the $L$-parameter of $\pi$, and $\zeta\in\fT\subset\fS_r$ the involution in Theorem~\ref{thm:key-intro}.
Now $\pi_i$ $(1\le i\le r)$ is of the form $\pi^{(1)}_{k_i,\lambda_i}, \pi^{(2)}_{k_i,\lambda_i}$ or $\tau_{k_i,\lambda_i}$ $(k_i\in\Z,\lambda_i\in\C)$. 
Write $\chi(z)=(z/|z|)^l|z|^{\eta}$ with $l\in\Z, \eta\in\C$.

From \cite[Theorem 1.3]{AV16} and the Langlands correspondence for $\GL_N(D)$, we see that the $L$-parameter of the contragredient $\pi^{\vee}$ equals the contragredient $\phi_\pi^{\vee}$ of the $L$-parameter of $\pi$.
By $\zeta\in\fT$ and Remark \ref{rem:parameter}, we have $-\lambda_{\zeta(i)}=\lambda_i-\eta$. 
Since $\pi_1,\pi_2,\ldots,\pi_r$ satisfy \eqref{eq:irrepR} or \eqref{eq:irrepH}, so do $\pi_{\zeta(1)}^{\vee},\pi_{\zeta(2)}^{\vee},\ldots,\pi_{\zeta(r)}^{\vee}$. 
Therefore $\pi_{\zeta(1)}^{\vee}\times\pi_{\zeta(2)}^{\vee}\times\cdots\times\pi_{\zeta(r)}^{\vee}$ is a standard module above $\pi^{\vee}$, and $(\pi_{\zeta(1)}^{\vee}\cdot\chi_{\GL_{n_1}(D)})\times(\pi_{\zeta(1)}^{\vee}\cdot\chi_{\GL_{n_2}(D)})\times\cdots\times(\pi_{\zeta(1)}^{\vee}\cdot\chi_{\GL_{n_r}(D)})$ is a standard module above $\pi^{\vee}\cdot\chi_G$. 
We see $\pi_i\cong \pi_{\zeta(i)}^{\vee}\cdot\chi_{\GL_{n_i}(D)}$ for any $1\le i\le r$ from $\zeta\in\fT$ and Remark \ref{rem:parameter}, and $\pi\cong\pi^{\vee}\cdot\chi_G$ follows.
\end{proof}

\section*{Appendix: Proof of Lemma \ref{lem:inv=coinv}}
Write $S(G)$ for the space of Schwartz functions on $G$, and $I(G)$ for the subspace consisting of $f\in S(G)$ with $\int f(g)dg=0$, where $dg$ denotes a right invariant Haar measure.  
\begin{lemma}[{\cite[Theorem 1.7]{CS21}}]\label{lem:coinv}
If $A\subset G$ intersects all connected components of $G$, then
\[\sum_{g\in G}(g-1)\pi
=I(G)V
=\fg V+\sum_{a\in A}(a-1)V.\]
\end{lemma}

\begin{proof}[Proof of Lemma \ref{lem:inv=coinv}]
Let $\hat{K}$ be the unitary dual of $K$. 
For $\delta\in\hat{K}$, write $\Pi(\delta)$ for the $\delta$-isotypic component of $\Pi$. 
Let $P_{\delta}(k)=\dim(\delta)\bar{\tr}_{\delta}(k)$ for $k\in K$. 
Then the action of the function $P_{\delta}$ is a continuous projection from $\Pi$ onto $\Pi(\delta)$.
Put $\Psi$ to be $\Pi$ or $\Pi_{\text{$K$-fin}}$, which contains $\Pi(\delta)$. 
The homomorphism $P_{\triv}$ induces a surjective continuous homomorphism $P$ from $\Psi_K$ onto $\Pi^K=\Pi(\triv)$. 
Since $P|_{\Pi^K}=\id_{\Pi^K}$, we have a direct sum decomposition $\Psi=\Pi^K\oplus \Ker(P_{\triv}|_{\Psi})$ as linear topological spaces. 

Let $Q$ be the composition of the inclusion $\Pi^K\to \Psi$ and the quotient $\Psi\to\Psi_K$. Then $Q$ is continuous and $P\circ Q=\id_{\Pi^K}$. 
Therefore what is left is to show $Q$ is surjective. 
More generally, let us prove $\overline{v}=\overline{P_{\triv}v}$ for any $v\in\Psi$, where $\bar{v}$ denotes the image of $v$ under the quotient $\Psi\to\Psi_K$. 

Fix $v\in\Psi$. 
When $\Psi=\Pi$, the series $\sum_{\delta\in\hat{K}} P_{\delta}v$ absolutely converges to $v$ by \cite[Theorem 4.4.2.1]{War72}. 
When $\Psi=\Pi_{\text{$K$-fin}}$, the same statement holds since $P_{\delta}v=0$ for almost all $\delta\in\hat{K}$. 
Put $\Gamma$ to be $K/K_0$, where $K_0$ denotes the neutral component of $K$. Then $\hat{\Gamma}$ can be regarded as a finite subset of $\hat{K}$. Define $v_1,v_2\in\Psi$ by 
\[v_1:=\sum_{\delta\in\hat{\Gamma}-\{\triv\}}P_{\delta}v,\quad 
v_2:=\sum_{\delta\in\hat{K}-\hat{\Gamma}}P_{\delta}v.\]

Let us consider the case $\delta\in\hat{\Gamma}-\{\triv\}$. 
We see $\sum_{\gamma\in\Gamma}\gamma x=0$ for $x\in\delta$ since $\delta^K=\{0\}$. 
Hence we have $x=(\#\Gamma)^{-1}\sum_{\gamma\in\Gamma}(1-\gamma)x$, and $\delta_K=0$. 
Therefore $\overline{v_1}=0$.

We next consider the case $\delta\in\hat{K}-\hat{\Gamma}$. 
In this case, the Casimir element $\Omega$ for $\fk$ acts on $\delta$ by a positive scalar $c(\delta)$. 
Moreover, we see $c(\delta)>1$ for almost all $\delta\in\hat{K}-\hat{\Gamma}$. 
Hence we can define 
\[u:=\sum_{\delta\in\hat{K}-\hat{\Gamma}}c(\delta)^{-1}P_{\delta}v\in\Psi\]
since $\sum_{\delta\in\hat{K}}P_{\delta}v$ absolutely converges in $\Psi$. 
By Lemma \ref{lem:coinv}, $\Omega\in\fk_{\C}U(\fk_{\C})$ and 
$\Omega u=v_2$,  we have $\overline{v_2}=0$. 
Therefore $\overline{v}=\overline{P_{\triv}v+v_1+v_2}=\overline{P_{\triv}v}$ and the proof is complete.
\end{proof}

\subsection*{Acknowledgements}
The first author is grateful to Hang Xue for bringing her attention to the theory of Schwartz homology.
M.S.\,was partially supported by JSPS Research Fellowship for Young Scientists No.20J00434 and Grant-in-Aid for Early-Career Scientists No.22K13891.
H.T.\,was partially supported by JSPS Research Fellowship for Young Scientists No.20J00024 and Grant-in-Aid for Early-Career Scientists No.23K12947.

\end{document}